\documentclass[12pt,reqno]{amsart}

\usepackage{amssymb}
\usepackage{amscd}
\usepackage{amsfonts}
\usepackage{setspace}
\usepackage{version}
\usepackage{mathrsfs}
\usepackage[colorlinks=true]{hyperref}

\usepackage{graphicx}

\newtheorem{theorem}{Theorem}[section]
\newtheorem{lemma}[theorem]{Lemma}
\newtheorem{proposition}[theorem]{Proposition}
\newtheorem{corollary}[theorem]{Corollary}

\theoremstyle{definition}

\newtheorem{definition}[theorem]{Definition}

\renewcommand{\leq}{\leqslant}
\renewcommand{\geq}{\geqslant}
\newcommand\SL{\operatorname{SL}}

\newcommand\GL{\operatorname{GL}}
\newcommand\Sp{\operatorname{Sp}}

\newcommand\adj{\operatorname{adj}}

\newcommand\Mat{\operatorname{Mat}}

\newcommand\Supp{\operatorname{Supp}}

\newcommand\Res{\operatorname{Res}}
\newcommand\PSL{\operatorname{PSL}}
\newcommand\U{\operatorname{U}}
\newcommand\Heis{\operatorname{H}}

\newcommand{\pmat}[4]{\begin{pmatrix} #1 & #2 \\ #3 & #4\end{pmatrix}}
\newcommand{\smallpmat}[4]{\big(\begin{smallmatrix} #1 & #2 \\ #3 & #4\end{smallmatrix}\big)}

\def\F{\mathbf{F}}

\def\R{\mathbf{R}}
\def\C{\mathbf{C}}
\def\Z{\mathbf{Z}}
\def\E{\mathbb{E}}

\def\Q{\mathbf{Q}}
\def\N{\mathbf{N}}
\def\T{\mathbf{T}}

\def\eps{\varepsilon}

\newcommand{\md}[1]{\ensuremath{(\operatorname{mod}\, #1)}}
\newcommand{\mdsub}[1]{\ensuremath{(\mbox{\scriptsize mod}\, #1)}}

\newcommand{\mdsublem}[1]{\ensuremath{(\mbox{\scriptsize \textup{mod}}\, #1)}}

\numberwithin{equation}{section}

\AtBeginDocument{
  \addtocontents{toc}{\Small}
}

\begin{document}

\title[Quadratic forms in 8 prime variables]{Quadratic forms in 8 prime variables}


\author{Ben Green}
\address{Mathematical Institute\\
Radcliffe Observatory Quarter\\
Woodstock Road\\
Oxford OX2 6GG\\
England}
\email{ben.green@maths.ox.ac.uk}
\thanks{The author is supported by a Simons Investigator grant and is grateful to the Simons Foundation for their continued support.}


\begin{abstract}
We give an asymptotic for the number of prime solutions to $Q(x_1,\dots, x_8) = N$, subject to a mild non-degeneracy condition on the homogeneous quadratic form $Q$. 

The argument initially proceeds via the circle method, but this does not suffice by itself. To obtain a nontrivial bound on certain averages of exponential sums, we interpret these sums as matrix coefficients for the Weil representation of the symplectic group $\Sp_8(\Z/q\Z)$. Averages of such matrix coefficients are then bounded using an amplification argument and a convergence result for convolutions of measures, which reduces matters to understanding the action of certain 12-dimensional subgroups in the Weil representation. Sufficient understanding can be gained by using the basic represention theory of $\SL_2(k)$, $k$ a finite field.

\end{abstract}

\maketitle

\tableofcontents

\section{Introduction}

Let $Q : \Z^8 \rightarrow \Z$ be a homogeneous quadratic form with integer coefficients, and let $N \in \Z$. We will study solutions to $Q(x_1,\dots, x_8) = N$ with the $x_i$ prime.

Throughout the paper, it is convenient to split $\Z^8 = \Z^4 \times \Z^4$ and to write 
\[ Q(x,y) = x^T a x + x^T b y + y^T c y\] where $a,b,c \in \Mat_4(\Z)$ with $a,c$ symmetric. Equivalently, 
\[ Q(x,y) = (x^T, y^T) \pmat{a}{b/2}{b^T/2}{c}\binom{x}{y}  .\]
Assume that $b$ is invertible. An important role will be played by the 4-by-4 matrix
\begin{equation}\label{delta-def} \Delta := 4b^{-1} a b^{-T} c - I .\end{equation}
This formally resembles the discriminant (of a form in two variables), only it is matrix-valued.

We turn now to the main theorem of the paper. Here, $\Lambda$ is the von Mangoldt function. For $q$ a positive integer, $\Lambda_{\Z/q\Z} : \Z/q\Z \rightarrow \C$ takes the value $q/\phi(q)$ when $(x,q) = 1$, and $0$ otherwise. As the notation suggests, this is the natural local variant of the von Mangoldt function. The normalisation is chosen so that the average value of $\Lambda_{\Z/q\Z}$ is 1. We abuse notation by writing $\Lambda_{\Z/p\Z}$ instead of $\Lambda_{\Z/p^n\Z}$ (the domain will always be clear from context). Finally, we write $\Lambda^{\otimes 4}(x) = \Lambda(x_1)\Lambda(x_2)\Lambda(x_3)\Lambda(x_4)$ for $x \in \Z^4$, and $\Lambda_{\Z/q\Z}^{\otimes 4}$ is defined analogously.

\begin{theorem}\label{mainthm}
Suppose that $Q$ is a quadratic form for which $\det ab \neq 0$ and $\Delta$ has four distinct eigenvalues which lie in $\overline{\Q} \setminus \{-1,0\}$. Then we have the expected local-global estimate for the number of solutions to $Q(x,y) = N$ in primes, namely for any real $A$ \[ \sum_{\substack{x,y \in [X]^4 \\ Q(x, y) = N}} \Lambda^{\otimes 4}(x) \Lambda^{\otimes}(y)  =  \mathfrak{S}(N) X^6  + O_{A,Q}(X^6 \log^{-A} X).\]
Here
\begin{equation} \mathfrak{S}(N) = \beta_{\infty} \prod_p \beta_p(N) \end{equation}
where $\beta_p(N) = \lim_{n \rightarrow \infty} \beta_{p,n}(N)$ is the $p$-adic density of solutions, where
\begin{equation}\label{beta-p}
\beta_{p,n}(N) := p^{-7n} \sum_{\substack{ x, y \in (\Z/p^n \Z)^4 \\  Q(x, y) \equiv 0 \mdsublem{p^n}}} \Lambda^{\otimes 4}_{\Z/p \Z}(x) \Lambda^{\otimes 4}_{\Z/p \Z}(y),
\end{equation}
and \[ \beta_{\infty} :=   \lim_{\delta \rightarrow 0} \frac{1}{2\delta} \mu_{\R^8} \{ x,y \in [0,1]^4 :  |Q(x,y) - \frac{N}{X^2}| \leq \delta\}\] is an archimedean measure of the density of \textup{(}positive, real\textup{)} solutions. Included in the statements is the fact that the limit in the definition of the $p$-adic density \eqref{beta-p} exists. 
\end{theorem}

Throughout the paper, we will say that a form $Q$ is \emph{generic} if it satisfies the conditions of this theorem, that is to say if $\det ab \neq 0$ and if $\Delta$ has four distinct eigenvalues in $\overline{\Q}$, and that neither $0$ nor $-1$ is one of these eigenvalues. The word generic is appropriate, since this condition holds for a Zariski-dense set of $(a,b,c)$ in the 36-dimensional parameter space where $a, c$ are symmetric. To see this, first note that the condition that $\Delta$ has distinct eigenvalues is (Zariski-)closed, by considering the resultant of the characteristic polynomial $\rho_{\Delta}(\lambda) := \det(\Delta - \lambda I)$ and its derivative $\rho'_{\Delta}(\lambda)$. The conditions that $\det ab = 0$, and that $\Delta$ has an eigenvalue $0$ or $-1$, are evidently closed conditions. Finally, these conditions are nontrivial (i.e. not always satisfied) as one can see by taking $a$ to be a diagonal matrix with distinct rational eigenvalues (not $0$ or $\frac{1}{4}$) and $b = c = I$. 

Note that $\Delta$ is not canonically associated to $Q$, being dependent on the splitting of variables into two sets of four. However, one may observe that if $a$ is invertible then
\[  \pmat{a}{b/2}{b^T/2}{c} = \pmat{1}{0}{A^T}{1} \pmat{B}{0}{0}{C} \pmat{1}{A}{0}{1}\] with $A = \frac{1}{2} a^{-1} b$, $B = a$, $C = \frac{1}{4} b^T a^{-1} b \Delta$, and so 
\begin{equation}\label{m-delta}\det Q := \det \pmat{a}{b/2}{b^T/2}{c} = 2^{-16} (\det b)^2 \det \Delta.\end{equation} (This is also true without the assumption that $a$ is invertible by a continuity argument.) Therefore the requirement that $0$ is not an eigenvalue of $\Delta$ is essentially an invariant of $Q$, more-or-less equivalent to $\det Q \neq 0$ (called the ``regular'' case in previous works such as \cite{zhao}). We do not expect any variant of our methods to handle the irregular case $\det Q = 0$, which essentially corresponds to forms in 7 or fewer variables. \vspace*{6pt}

\emph{Previous results.} Jianya Liu \cite{liu} handled generic quadratic forms in 10 prime variables.  Lilu Zhao \cite{zhao} subsequently handled all regular quadratic forms in 9 variables. These works use fairly classical forms of the Hardy-Littlewood circle method. For reasons we will go into later, 9 variables appears to be the limit of what any such method can give, and so far as I am aware results in 8 or fewer variables are known only for particular types of form with some degenerate and/or diagonal behaviour. For diagonal forms,  5 variables suffice by work of Hua \cite{hua}. If one is content with almost-primes instead of primes, the number of variables can be reduced to $3$: see \cite{bgs} and subsequent works.
\vspace*{6pt}

\emph{Future work.} In our main theorem we imposed conditions on $Q$, namely that $\det ab \neq 0$ and the matrix $\Delta$ has four distinct eigenvalues in $\overline{\Q} \setminus \{0, -1\}$. Whilst these are not especially severe restrictions, it nonetheless seems to be of interest to weaken them as far as possible, and we intend to address this in future work. 

There are at least three paths to pursue in this direction. First, there are 35 essentially different ways to split $8$ variables into two groups of $4$, which one would expect to lead to $\Delta$s with different properties. Second, many of the arguments of the paper can be modified to work in more degenerate situations. Finally, in some highly degenerate situations Theorem \ref{mainthm} can be established by classical methods such as those in \cite{zhao}. One would expect this to be the case when $\det b = 0$ for all splittings of the 8 variables (the case of low ``off-diagonal rank'').  We anticipate this to be a somewhat tedious endeavour, with all the main ideas already present in this paper and \cite{zhao}.

On a different matter, allowing linear terms in $Q$ (that is, nonhomogeneous quadratics) is probably possible but seems to require a fairly significant modification of the method, and we will not attempt this here.\vspace*{6pt}

\emph{Notation.} Most of our notation is fairly standard. We write $e(t) = e^{2\pi i t}$, and for $q$ a positive integer we write $e_q(x) := e^{2\pi i x/q}$. If $S$ is a finite set and $F : S \rightarrow \C$ a function, we write $\E_{x \in S} F(x)$ to mean the average of $F$ over $S$. We write $\T = \R/\Z$, and we write $\Vert x \Vert_{\T}$ for the distance from $x$ to the nearest integer. We write $[X]$ to denote the discrete interval $\{1,\dots, X\}$, and $[0,X]$ for the continuous interval $\{ x : 0 \leq x \leq X \} \subset \R$.

If $G$ is a finite group then we write $\ell^2(G)$ for the vector space of all functions $f : G \rightarrow \C$ together with the inner product $\langle f_1, f_2\rangle := \E_{x \in G} f_1(x) \overline{f_2(x)}$ and the associated norm $\Vert f \Vert_2 := \sqrt{\langle f, f\rangle}$. Later on in the paper we will also define norms of probability measures,  and we caution the reader that there we will use a different normalisation.  

If $V$ is a Hermitian inner product space (such as $\ell^2(G)$) then we write $\U(V)$ for the group of unitary transformations of $V$.

For $p$ a prime, we will freqently encounter the group $\Z/p\Z$. When this arises as a group or a ring, we will write it $\Z/p\Z$, but when it is important that it is a field, we will write $\F_p$. This may seems slightly eccentric, but it does not seem stylistically correct to talk about homomorphisms from $\Z/q\Z$ to $\F_p$ (when $p | q$) and nor does it seem right to discuss the field $(\Z/p\Z)(\theta)$ or the algebraic closure $\overline{\Z/p\Z}$. At times the distinction is somewhat arbitrary. 

Throughout the paper we abuse notation in certain standard ways which should not cause any confusion. For example, we also write $Q(x,y)$ for the quadratic form $x^T a x + x^T b y + y^T c y$ over $\Z/q\Z$, by which we mean that $x,y \in (\Z/q\Z)^4$ and $a,b$ and $c$ are to be considered $\md{q}$. Slightly more subtly, we also consider $\Delta$ (defined in \eqref{delta-def} as an element of $\Mat_4(\Q)$) as an element of $\Mat_4(\Z/p\Z)$, which makes sense provided $p \nmid \det b$. Similarly, we consider a certain 8-by-8 symplectic matrix $g$ (defined in \eqref{g-def-prelim} below) as an element of $\Sp_8(\Z/q\Z)$ for squarefree $q$, which again will make sense provided no prime factor of $q$ divides $\det b$.

Finally, we regard the quadratic form $Q$ as fixed throughout the paper and will not explicitly indicate dependence on $Q$ in asymptotic notation such as $\ll$ or $O()$.

\vspace*{6pt}

\emph{Acknowledgements.} It is a pleasure to thank Emmanuel Breuillard, Charlotte Chan, Tom Fisher and Bal\'azs Szendr\H{o}i for helpful correspondence related to this work and earlier versions of it, and Roger Baker and James Maynard for discussions which introduced me to the problem in around 2014. The author is a Simons Investigator and is very grateful to the Simons Foundation for their continued support.

\section{Outline of the argument}

The initial steps of the argument proceed in the classical fashion using the circle method, which we set up in Section \ref{sec3}. We introduce the exponential sum 
\[ S(\theta) := \sum_{x, y\in [X]^4}  \Lambda^{\otimes 4}(x)\Lambda^{\otimes 4}(y) e(\theta Q(x,y)),\]
where $\Lambda^{\otimes 4}(x)$ is shorthand for $\Lambda(x_1) \Lambda(x_2)\Lambda(x_3)\Lambda(x_4)$, and of course $\Lambda$ denotes the von Mangoldt function. Then by orthogonality we have
\[ \sum_{\substack{x,y \in [X]^4 \\ Q(x, y) = N}} \Lambda^{\otimes 4}(x)\Lambda^{\otimes 4}(y) = \int_{\T} S(\theta) e(-\theta N) d\theta.\]
We divide $\T$ into the major arcs $\mathfrak{M}$ (roughly, the set of $\theta$ within distance $\sim X^{-2}$ from a rational $\frac{a}{q}$ with $q \ll \log^{O(1)} X$) and the minor arcs $\mathfrak{m}$. The major arcs give the main term in the asymptotic, and the analysis of them is entirely classical. We give this analysis in Section \ref{sec4}, referring to \cite{zhao} for the details when possible.

For the minor arcs, the fact that we are discussing primes is essentially irrelevant and the same arguments work with $\Lambda^{\otimes 4}(x)\Lambda^{\otimes 4}(y)$ replaced by $F_1(x) F_2(y)$ for any reasonably bounded functions $F_1, F_2$. We in fact divide the minor arcs into two sets $\mathfrak{m}_1$ and $\mathfrak{m}_2$, with $\mathfrak{m}_1$ being points not too close to a rational and $\mathfrak{m}_2$ being the points very close to a rational (but with moderately large denominator). The precise definitions are given at the start of Section \ref{sec3}. The treatment of the integral over $\mathfrak{m}_1$ uses diophantine approximation arguments standard in the area, and is given in Section \ref{sec5}.

The treatment of the minor arcs $\mathfrak{m}_2$ is the heart of the paper. One may reduce to considering actual rational points $\frac{r}{q}$, with $q > \log^C X$ moderately large, and one is then led naturally led to look at exponential sums of the form
\begin{equation}\label{tf1f2} T_{f_1, f_2}(r) := q^2 \E_{x,y \in (\Z/q\Z)^4} f_1(x) f_2(y) e_q(r Q(x,y)),\end{equation} where here $r \in (\Z/q\Z)^*$.
There is a ``trivial'' upper bound of $1$ for such sums when $\Vert f_1 \Vert_{2}, \Vert f_2 \Vert_2 \leq 1$, which turns out to be (just) not good enough for the purposes of bounding the integral over $\mathfrak{m}_2$. However, any improvement of it by a factor $q^{-\delta}$ would suffice.

Unfortunately, there is no such improvement: the trivial bound is best possible. However, by a less wasteful reduction we can arrange things so that we consider instead the averages
\begin{equation}\label{t-avg} \frac{1}{\phi(q)}\sum_{r \in (\Z/q\Z)^*} |T_{f_1, f_2}(r)|.\end{equation}
Again, a saving of $q^{-\delta}$ over the trivial bound of $1$ would be enough. 

We incorporate some additional tricks which allow us to restrict attention to the case $q$ squarefree and without very small prime factors, two features which are vital in our later arguments. The details of these reductions are given in Section \ref{sec6}.

The remainder of the paper is devoted to establishing a nontrivial bound of the required strength for averages \eqref{t-avg}. To make progress on this problem, we interpret the exponential sums $T_{f_1, f_2}(r)$ as matrix coefficients $\langle f_1, \rho(g^{(r)}) f_2\rangle$, where here $\rho : \Sp_8(\Z/q\Z) \rightarrow \U(\ell^2((\Z/q\Z)^4))$ is a certain unitary representation of the symplectic group $\Sp_8$ over $\Z/q\Z$ called the Weil representation. After a brief introduction to the symplectic group and the Weil representation, we give this interpretation in Section \ref{sec7}. Whilst the theory of the Weil representation is well-known over $\R$ and somewhat well-known over finite fields, we do not know of a good source for the theory we need over $\Z/q\Z$, so we must develop some of this ourselves. This is fairly straightforward given the finite field statements, and is done in Appendix \ref{weil-app}.

The elements $g^{(r)} \in \Sp_8(\Z/q\Z)$ are what we call ``dilates'' of a single element $g = g^{(1)}$ given by the formula
\begin{equation}\label{g-def-prelim} g := \pmat{-2b^{-T} c}{b^{-T}}{4ab^{-T}c - b}{-2ab^{-T}}.\end{equation} The \emph{dilate} of $\smallpmat{A}{B}{C}{D} \in \Sp_8(\Z/q\Z)$ by $r \in (\Z/q\Z)^*$ is $\smallpmat{A}{r^{-1} B}{rC}{D}$; this is in fact an action of $(\Z/q\Z)^*$ by automorphisms, as may be easily checked. 

One is therefore led to the question of bounding an average of matrix coefficients $|\langle f_1, \rho(g^{(r)}) f_2\rangle|$, where $r$ ranges over $(\Z/q\Z)^*$. 

In Section \ref{sec8} we supply a general tool for bounding averages of matrix coefficients, in principle applicable to any unitary representation $\psi$ of any finite group $G$. This allows one to bound an average \[ \int_G |\langle f_1, \psi(x) f_2\rangle| d\mu(x),\] where $\mu$ is a probability measure on $G$, when two conditions are satisfied:
\begin{enumerate}
\item (convergence to uniform measure) Some symmetrised convolution power $\mu^{\circ} \ast \mu \ast \mu^{\circ} \ast \mu \cdots$ of bounded order should be close to the uniform measure on a subgroup $H \leq G$;
\item (quasirandomness) $\psi|_H$ has no low-dimensional irreducible components.
\end{enumerate}
We wish to apply this tool with $G = \Sp_8(\Z/q\Z)$, $\psi = \rho$ being the Weil representation, and $\mu$ being the uniform measure on the $\phi(q)$ points $g^{(r)}$, $r \in (\Z/q\Z)^*$. To do this we need to establish the convergence and quasirandomness properties.

The task (1) of showing that (symmetrised) convolution powers of $\mu$ converge to a uniform measure on a subgroup suggests the literature on the affine sieve, expanders and general measure convolutions in groups of Lie type, in particular the work of Varj\'u \cite{varju} which provides results in the appropriate generality. Some variant of this can probably be made to work in our context. However, our particular measure $\mu$ has a rather algebraic definition, being parametrised by (very simple) rational functions and we are able to offer an alternative approach using the Lang-Weil estimate. This is inspired by a blog post of Tao \cite{tao-spectral}, giving an alternative proof (inspired by model-theoretic work of Pillay and Starchenko \cite{pillay-starchenko} and unpublished notes of Hrushovski) of his own algebraic regularity lemma \cite{tao-algebraic-regularity}. This may be of independent interest, though we only develop it in the specific setting of interest to us here. This allows one to demonstrate rapid convergence of (symmetrised) powers of $\mu$ to the uniform measure on the group $\Gamma_q$ generated by the elements $g^{-(r)} g^{(s)}$, without knowing \emph{a priori} what this group is. The arguments may be found in Section \ref{unif-conv-power}.

The remaining task (2) is to establish the quasirandomness property for $\rho|_{\Gamma_q}$. It is easy to see that $\Gamma_q \cong \prod_{p | q} \Gamma_p$, and so by using the basic theory of tensor product representations it turns out to be enough to understand the case $q = p$ prime. First, we identify $\Gamma_p$ explicitly. I was initially under the impression that the elements $g^{-(r)} g^{(s)}$ might generically generate the whole of $\Sp_8(\Z/p\Z)$ (which has size $\sim p^{36}$), on the grounds that there is no immediately evident reason why they should not, and for the analogous situation in $\Sp_2(\Z/p\Z)$ this is true. However, it turns out that this is not the case, and that $\Gamma_p$ is (generically) a group of size $\sim p^{12}$, a conjugate (in $\GL_8(\Z/p\Z)$) of $\SL_2(\F_p[\Delta])$, where here $\Delta$ is the matrix discriminant given in \eqref{delta-def}. Establishing this takes some work, involving calculations in $\SL_2$ together with applications of lemmas of Goursat and Ribet on subgroups of direct products. A number of facts about $\SL_2$ of a finite field are required here, and these are collated in Appendix \ref{sl2-app}. These tasks are accomplished in Section \ref{sec10}.

With $\Gamma_p$ identified explicitly, we turn to the quasirandomness property itself. It is essentially automatic from the representation theory of $\SL_2(\F_{p^n})$ that if $\rho|_{\Gamma_p}$ has an irreducible component of small degree, then this component must be the trivial representation: that is, $\Gamma_p$, acting via the Weil representation on $\ell^2((\Z/p\Z)^4)$, would have a nontrivial fixed vector. The final task of the paper, then, is to rule this out. We do this in Section \ref{sec11} using rather direct and explicit (that is, not using any representation theory) arguments.

\section{The circle method}\label{sec3}

In this section we describe the basic setup of the circle method. As is typical in problems of this type we will be aiming for error terms in our main theorem of $O_A(X^6 \log^{-A} X)$, for an arbitrary positive real number $A$.  Fix such an $A$, without loss of generality $A \geq 10$, and set \begin{equation}\label{q-def} M := \log^{C_1} X, \quad M' := \log^{C_2} X,\end{equation} where $C_1 := 10 A$ and $C_2  := 10^5 A\delta^{-1}$, where $\delta$ is the exponent appearing in Proposition \ref{zq-main-prop} below (these choices are by no means optimal, but this is inconsequential).

Set \begin{equation}\label{K-def} K := 8 \max_{ij} |b_{ij}|,\end{equation} where the $b_{ij}$ are the entries of the 4-by-4 matrix $b$. Thus $K$ is a constant depending only on the quadratic form $Q$.  For $q \in \N$ and for $r \in (\Z/q\Z)^*$, denote 

\begin{equation}\label{irq-def} I_{r,q} := \{\theta \in \T: |\theta - \frac{r}{q}| \leq \frac{M}{X^2}\} \end{equation}
and
\begin{equation}\label{irq-tilde}  \tilde I_{r,q} :=  \{ \theta \in \T : |\theta - \frac{r}{q}| \leq \frac{1}{KqX} \}.\end{equation} 
Define the \emph{major arcs}
\begin{equation}\label{major-arc-def} \mathfrak{M} := \bigcup_{\substack{q \leq M' \\ r \in (\Z/q\Z)^*}} I_{r,q},\end{equation} and set
\begin{equation}\label{minor-def}  \mathfrak{m}_1 := \bigcup_{\substack{q \leq KX \\ r \in (\Z/q\Z)^*}} (\tilde I_{r,q} \setminus I_{r,q}), \quad  \mathfrak{m}_2 := \bigcup_{\substack{M' < q \leq KX \\ r \in (\Z/q\Z)^*}} I_{r,q}.\end{equation}

\begin{lemma}\label{minor-major}
We have $\mathfrak{M} \cup \mathfrak{m}_1 \cup \mathfrak{m}_2 = \T$.
\end{lemma}
\begin{proof}
By Dirichlet's theorem on diophantine approximation, 
\[ \T = \bigcup_{\substack{q \leq KX \\ r \in (\Z/q\Z)^*}} \tilde I_{r,q}.\]
The result then follows immediately.
\end{proof}

Set
\begin{equation}\label{s-def} S(\theta) := \sum_{x, y \in [X]^4}  \Lambda^{\otimes 4}(x)\Lambda^{\otimes 4}(y) e(\theta Q(x,y)),\end{equation}
where, recall, $\Lambda^{\otimes 4}(x)$ is a convenient shorthand for $\prod_{i=1}^4 \Lambda(x_i)$. Then by orthogonality we have
\begin{equation}\label{orth} \sum_{\substack{x,y \in  [X]^4 \\ Q(x,y) = N}} \Lambda^{\otimes 4}(x)\Lambda^{\otimes 4}(y) = \int_{\T} S(\theta) e(- \theta N) d\theta.\end{equation}

We evaluate this by considering the contributions to the integral from $\mathfrak{M}, \mathfrak{m}_1, \mathfrak{m}_2$ separately. The major arcs $\mathfrak{M}$ give the main term in the asymptotic, as the following result shows.

\begin{proposition}[Major arcs]\label{major-arc-prop}
Suppose that $Q$ is regular, that is to say $\smallpmat{a}{b/2}{b^T/2}{c}$ is nonsingular. Then we have
\[ \int_{\mathfrak{M}} S(\theta) e(- \theta N) d\theta = \mathfrak{S}(N) X^6 + O_A(X^6 \log^{-A} X),\] where the singular series $\mathfrak{S}(N)$ is as described in Theorem \ref{mainthm}.
\end{proposition}

We will prove this in the next section using classical methods, referring to \cite{zhao} for most of the details.

Now we turn to the minor arcs $\mathfrak{m}_1$ and $\mathfrak{m}_2$. Here, as previously remarked, the fact that we are dealing with primes and the von Mangoldt function is essentially irrelevant. For any functions $F_1, F_2 : [X] \rightarrow \C$ we introduce the sums 
\begin{equation}\label{sf1f2-def} S_{F_1, F_2}(\theta) :=  \sum_{x, y\in [X]^4}  F_1(x)F_2(y) e(\theta Q(x,y)).\end{equation}

\begin{proposition}[Minor arcs $\mathfrak{m}_1$]\label{minor-arcs-m1}
Suppose that $\det b \neq 0$. Then we have
\[ \int_{\mathfrak{m}_1} |S_{F_1, F_2}(\theta)| d\theta \ll_A X^6 \log^{-A-8} X,\] uniformly for all $1$-bounded functions $F_1, F_2$.
\end{proposition}

We will prove this in Section \ref{sec5}, using diophantine approximation arguments typical of the circle method.

\begin{proposition}[Minor arcs $\mathfrak{m}_2$]\label{minor-arcs-m2}
Suppose that $Q$ is generic \textup{(}that is, $\Delta$ has four distinct eigenvalues in $\overline{\Q} \setminus \{0, -1\}$\textup{)}. Then we have
\[ \int_{\mathfrak{m}_2} |S_{F_1, F_2}(\theta)| d\theta \ll_A X^6 \log^{-A-8} X,\] uniformly for all $1$-bounded functions $F_1, F_2$.
\end{proposition}

This proof of this, which is a substantial undertaking, contains the new ideas of the paper and occupies the remaining sections.

Let us conclude this section by remarking that Propositions \ref{major-arc-prop}, \ref{minor-arcs-m1} and \ref{minor-arcs-m2} easily combine to establish Theorem \ref{mainthm}. Indeed,  by \eqref{orth} and Proposition \ref{major-arc-prop} we have
\begin{align*}   \big| \sum_{\substack{x,y \in  [X]^4 \\ Q(x,y) = N}} \Lambda^{\otimes 4}(x)&\Lambda^{\otimes 4}(y) - \mathfrak{S}(N) X^6 \big| \\ &  \leq O_A(X^6\log^{-A} X) + |\int_{\T \setminus \mathfrak{M}} S(\theta) e(-\theta N) d\theta|.\end{align*}
By the triangle inequality and Lemma \ref{minor-major}, the second term on the right is bounded above by
\[ \int_{\mathfrak{m}_1} |S(\theta)| d\theta + \int_{\mathfrak{m}_2} |S(\theta)| d\theta.\]
By Propositions \ref{minor-arcs-m1} and \ref{minor-arcs-m2} (taking $F_1 = F_2 = (\log X)^{-4} \Lambda^{\otimes 4}$), both of these terms are bounded by $O_A(X^6 \log^{-A} X)$.

\section{The major arcs}\label{sec4}

In this section we establish Proposition \ref{major-arc-prop}. The argument is very classical and in fact large portions of it may be simply quoted from \cite{zhao}. For this part of the argument, similar results hold with as few as 5 variables. Define
\begin{equation}\label{zhao-c} C(q,r) := \sum_{\substack{ x,y \in (\Z/q\Z)^4 \\ (x,q) = (y,q) = 1}} e_q(r Q(x,y)),\end{equation}
\begin{equation}\label{zhao-b} B_{Q,N}(q) := \frac{1}{\phi(q)^8} \sum_{r \in (\Z/q\Z)^*} C(q,r) e_q(-rN),\end{equation}
\begin{equation}\label{zhao-s} \mathfrak{B}_{Q,N} := \sum_{q = 1}^{\infty} B_{Q,N}(q),\end{equation}
\begin{equation}\label{zhao-i} I(t) := \int_{u,v \in [0, X]^4} e(tQ(u,v)) dudv\end{equation} and
\begin{equation}\label{zhao-j} \mathfrak{J}_{Q,N}(X) := \int^{\infty}_{-\infty} I(t) e(- Nt).\end{equation}
These are the same definitions as those in \cite[Section 3]{zhao}, with some notational substitutions (Zhao's $f, t, \beta, \mathfrak{S}(f,t)$ become our $Q, N, t, \mathfrak{B}_{Q,N}$ respectively). Also, we notate our quadratic forms using two variables $x, y$. Definitions like these will be familiar to anyone with knowledge of the circle method. The following is \cite[Lemma 3.6]{zhao}.
\begin{lemma}[Major arcs]\label{major-arc-lem}
We have 
\begin{equation}\label{major-arcs} \int_{\mathfrak{M}} S(\theta) e(- \theta N) d\theta = \mathfrak{B}_{Q,N} \mathfrak{J}_{Q,N}(X) + O_A(X^6 \log^{-A} X).\end{equation}
\end{lemma}
\emph{Remarks.} Our choice of $C_1$ and $C_2$ in the definition \eqref{q-def} of $M, M'$ ensures that our major arcs are amply wide enough that the error term in \cite[Lemma 3.6]{zhao} is $O_A( X^6 \log^{-A} X)$. The only consequence of taking the major arcs this wide is that the choice of exponent in the error term of the Siegel-Walfisz theorem towards the end of \cite[Section 3]{zhao} (which is, in any case, not made explicit there) must be larger.

There is one further inconsequential difference between our setup and that in \cite{zhao}. In \cite{zhao} the major arc about $\frac{r}{q}$ has width $\frac{M}{qX^2}$, whereas we have taken the width to be $\frac{M}{X^2}$.  The only other tiny change required is in (3.17), (3.19) of \cite{zhao} where the integrals should be taken over our slightly longer range $|\beta| \leq M/X^2$ (which actually helps slightly). 

To reconcile this with Proposition \ref{major-arc-prop} we must express $\mathfrak{B}_{Q.N}$ and $\mathfrak{J}_{Q,N}(X)$ in terms of the local densities $\beta_p, \beta_{\infty}$, whose definitions are given in the statement of Theorem \ref{mainthm}. This is again a standard endeavour, but it is not done in Zhao's paper so we give brief details now.

Recall the definition \eqref{beta-p} of $\beta_{p,n}(N)$. By orthogonality, we have
\[ \beta_{p, n}(N) = p^{-8n} \sum_{r \in \Z/p^n \Z}\sum_{x, y \in (\Z/p^n \Z)^4} \Lambda_{\Z/p\Z}^{\otimes 4}(x) \Lambda_{\Z/p\Z}^{\otimes 4}(y) e_{p^n}(r (Q(x,y) - N)).\]
In the sum over $r$, write $r = p^{n - j} r'$ with $(r', p) = 1$. One may check that the contribution from a particular $j$ is $B_{Q,N}(p^j)$, and so 
\[ \beta_{p, n}(N) = \sum_{j = 0}^n B_{Q,N}(p^j).\] Taking the limit as $n \rightarrow \infty$ gives
\[ \beta_p(N) = \sum_{j = 0}^{\infty} B_{Q,N}(p^j).\]
Finally, since $B_{Q,N}(q)$ is a multiplicative function of $q$ (see \cite[Lemma 3.1]{zhao}) we have
\begin{equation}\label{non-arch} \prod_{p} \beta_p (N) = \sum_{q = 1}^{\infty} B_{Q,N}(q) = \mathfrak{B}_{Q,N}.\end{equation}
There are, of course, convergence issues to be dealt with here, but these are fully fleshed out in \cite[Lemma 3.4]{zhao}.

To handle the archimedean factor $\beta_{\infty}$, we proceed is as follows (we leave detailed analytic justifications to the reader). For $\eps > 0$, set
\[ f_{\eps}(w) := \frac{1}{2\eps} \mu_{\R^4} \{ u,v \in \R^4 : u,v \in [X]^4, w - \eps \leq Q(u,v) \leq w + \eps\}.\]
Fourier inversion then gives
\[ f_{\eps}(N) = \int^{\infty}_{-\infty} \hat{f}_{\eps}(t) e(Nt) dt.\]
However, 
\[ I(t) = \lim_{\eps \rightarrow 0} \hat{f}_{\eps}(-t).\]
Taking limits as $\eps \rightarrow 0$ (and substituting $x := \frac{u}{X}$, $y := \frac{v}{X}$, $\delta := \frac{\eps}{X^2}$ and using the homogeneity of $Q$) gives 
\begin{equation}\label{jqn} \mathfrak{J}_{Q,N}(X) = \beta_{\infty} X^6.\end{equation}
Substituting \eqref{non-arch} and \eqref{jqn} into Lemma \ref{major-arc-lem} gives Proposition \ref{major-arc-prop}.

\section{Minor arcs: the integral over \texorpdfstring{$\mathfrak{m}_1$}{m1}}\label{sec5}

In this section we prove Proposition \ref{minor-arcs-m1}. The reader may wish to recall the definitions of $I_{r,q}$, $\tilde I_{r,q}$ and $\mathfrak{m}_1$, which are \eqref{irq-def}, \eqref{irq-tilde} and \eqref{minor-def} respectively.

\begin{proof}[Proof of Proposition \ref{minor-arcs-m1}]
Observe that $\tilde I_{r,q} \setminus I_{r,q}$ is empty if $q >X/KM$, and therefore
\begin{equation}\label{m1-meas} \mu(\mathfrak{m}_1) \leq \sum_{q \leq X/KM} \sum_{r \in (\Z/q\Z)^*} \frac{2}{Kq X} <\frac{1}{M}.\end{equation}

Write
\begin{equation}\label{sf1f2-rewrite} S_{F_1, F_2}(\theta) := \sum_{x,y \in [X]^4} F'_1(x) F'_2(y) e(\theta x^T b y)\end{equation} where $F'_1(x) := F_1(x) e(\theta x^T a x)$, $F'_2(y) := F_2(y) e(\theta  y^T c y)$.
By Cauchy-Schwarz, 
\[ |S_{F_1, F_2}(\theta)| \leq X^2 \big(\sum_{y,y' \in [X]^4} F'_2(y) \overline{F'_2(y')} \sum_{x \in [X]^4} e(\theta x^T b (y - y'))\big)^{1/2}.\]

By the standard estimate
\[ |\sum_{x \in [X]} e(\xi x)| \ll \min(X, \Vert \xi \Vert_{\T}^{-1})\] and since $\Vert F_2 \Vert_{\infty} \leq 1$ it follows that
\begin{equation}\label{sf1-use} |S_{F_1, F_2}(\theta)| \ll X^2 \big(\sum_{y,y' \in [X]^4}  \prod_{j = 1}^4 \min(X, \Vert \theta (b(y - y'))_j \Vert_{\T}^{-1})\big)^{1/2}.\end{equation}
Now the image of $[X]^4 \times [X]^4$ under the map $(y, y') \mapsto b^T (y - y')$ is contained in the box $[-\frac{1}{2}KX, \frac{1}{2}KX]^4$ and the fibres are of size at most $X^4$ (since $b$ is nonsingular; recall also from the definition \eqref{K-def} that $K = 8 \max_{ij} |b_{ij}|$). It follows from \eqref{sf1-use} that 
\begin{equation}
|S_{F_1, F_2}(\theta)| \ll  X^{4} \big(  \sum_{\substack{h \in \Z \\ |h| \leq KX/2}}  \min(X, \Vert \theta h \Vert_{\T}^{-1}) \big)^2.\label{first-est}
\end{equation}

Suppose now that $\theta \in \tilde I_{r,q} \setminus I_{r,q}$, thus 
\[ \theta = \frac{r}{q} + \eta \quad \mbox{where} \quad \frac{M}{X^2} < |\eta| \leq \frac{1}{KqX}.\] Foliating into progression modulo $q$ we have \begin{align}\nonumber  \sum_{|h| \leq KX/2} \min(X, & \Vert \theta h \Vert_{\T}^{-1}) \\ & = \sum_{b \mdsub q} \sum_{\substack{|h| \leq KX/2 \\ h \equiv b \mdsub{q}}} \min(X, \Vert (\frac{r}{q} + \eta) h) \Vert_{\T}^{-1}).\label{t-sum} \end{align}
We evaluate the contributions from $b = 0$ and $b \neq 0$ separately. If $b \neq 0$, $h \equiv b \md{q}$ and $|h| \leq KX/2$ then 
\[ \Vert (\frac{r}{q} + \eta) h \Vert_{\T} \geq \Vert \frac{rb}{q} \Vert_{\T} - \frac{KX}{2}|\eta|\geq \Vert \frac{rb}{q}\Vert_{\T} - \frac{1}{2q}.\] Thus, if $b \neq 0$, 
\begin{equation}\label{b-not-zero} \sum_{\substack{|h| \leq KX/2 \\ h \equiv b \mdsub{q}}} \min(X, \Vert (\frac{r}{q} + \eta) h) \Vert_{\T}^{-1}) \ll \frac{X}{q} \big(\Vert \frac{rb}{q}\Vert_{\T} - \frac{1}{2q}\big)^{-1}.\end{equation} (Recall here that $q \leq KX$, so the number of terms in the sum over $h$ is indeed $\ll X/q$.) Now as $b$ ranges over $(\Z/q\Z) \setminus \{0\}$, so does $rb$.  Thus
\[ \sum_{\substack{b \mdsub{q} \\ b \neq 0}} \big(\Vert  \frac{rb}{q}\Vert_{\T} - \frac{1}{2q}\big)^{-1} = \sum_{\substack{s \mdsub{q} \\ s \neq 0}} \big(\Vert \frac{s}{q}\Vert_{\T} - \frac{1}{2q}\big)^{-1} \ll q \log q.\] 
Substituting into \eqref{b-not-zero}, we see that the contribution to the right-hand side of \eqref{t-sum} from the terms with $b \neq 0$ is $\ll X \log q = O(X\log X)$.

Now we look at the contribution to the right-hand side of \eqref{t-sum} from $b = 0$. Making the substitution $h = kq$, this is
\begin{equation}\label{b0-contrib} \sum_{|k| \leq KX/2q} \min(X, \Vert \eta k q \Vert_{\T}^{-1}).\end{equation}
We have
\[ |\eta k q| \leq \frac{1}{KqX} \cdot \frac{KX}{2q} \cdot q < \frac{1}{2},\] so $\Vert \eta k q \Vert_{\T} = |\eta k q|$. Therefore \eqref{b0-contrib} is
\[ \sum_{|k| \leq KX/2q} \min(X, |\eta k q |^{-1})  \leq X + \sum_{0 < |k| < KX/2q} |\eta k q|^{-1}  \ll X + \frac{1}{\eta q} \log X.\]
Substituting these bounds for $b \neq 0$ and $b = 0$ into \eqref{t-sum}, we obtain
\[ \sum_{|h| \leq KX/2} \min(X, \Vert \theta h \Vert_{\T}^{-1}) \ll (X  + \frac{1}{\eta q} )\log X.\]
Substituting into \eqref{first-est} gives, for $\theta = \frac{r}{q} + \eta \in \tilde I_{r,q} \setminus I_{r,q}$, 
\[ |S_{F_1, F_2}(\theta) | \ll X^6 \log^2 X + \frac{X^4 \log^2 X }{\eta^2 q^2}.\]
To complete the proof of Proposition \ref{minor-arcs-m1}, we need to integrate this estimate over $\theta \in \mathfrak{m}_1$, that is to say over all $\tilde I_{r,q} \setminus I_{r,q}$ with $q \leq KX$ and $r \in (\Z/q\Z)^*$. The contribution from the first term $X^6 \log^2 X$ is at most $X^6 M^{-1} \log^2 X$ by \eqref{m1-meas}. The contribution from the second term is
\begin{align*}  \ll X^4 \log^2 X & \sum_{q \leq KX} \frac{1}{q^2} \sum_{r \in (\Z/q\Z)^*} \int_{M/X^2}^{1/KqX} \eta^{-2} d\eta \\ & \leq X^4 \log^2 X \sum_{q \leq KX} \frac{1}{q} \frac{X^2}{M} \ll X^6 M^{-1}  \log^3 X.\end{align*}
Recalling that $M = \log^{C_1} X$ with $C_1 = 10A$, this completes the proof.
\end{proof}

\section{The integral over \texorpdfstring{$\mathfrak{m}_2$}{m2} -- first reductions}\label{sec6}

We now begin the lengthy task of establishing Proposition \ref{minor-arcs-m2}. Once again, the reader may wish to begin by recalling the pertinent definitions, which are those of $I_{r,q}$ (see \eqref{irq-def}), $\mathfrak{m}_2$ (see \eqref{minor-def}) and $S_{F_1, F_2}(\theta)$ (given in \eqref{sf1f2-def}). 

At the heart of our analysis will be certain complete exponential sums $T_{f_1, f_2}$. Let $q$ be a positive integer, and suppose that $f_1, f_2 : (\Z/q\Z)^4 \rightarrow \C$. Define
 \begin{equation}\label{t-def} T_{f_1, f_2}(r) :=  q^2\E_{x,y \in (\Z/q\Z)^4}  f_1(x) f_2(y) e_q(r Q(x, y) ).\end{equation}  
 \emph{Remark.} Of course, $T$ is also depends on $q$, but we omit explicit mention of this from the notation. There should not be any danger of confusion.
 For fixed $r$ and general $f_1, f_2$ we have the following bound.
 
 \begin{lemma}\label{triv-bd}
Suppose that $f_1, f_2 : (\Z/q\Z)^4 \rightarrow \C$. Suppose that $\det b \neq 0$. Then for any $r \in (\Z/q\Z)^*$ we have \begin{equation}\label{triv-T}  |T_{f_1, f_2}(r)| \leq (q, \det b)^2 \Vert f_1 \Vert_2 \Vert f_2 \Vert_2.\end{equation} \end{lemma}
\begin{proof}
Modifying $f_1(x)$ to $f_1(x) e_q(r x^T a x)$ and $f_2(y)$ to $f_2(y) e_q(ry^T c y)$, it suffices to show that 
\[ \E_{x,y \in (\Z/q\Z)^4} f_1(x) f_2(y) e_q( rx^T b y) \leq (q, \det b)^2 q^{-2} \Vert f_1 \Vert_2 \Vert f_2 \Vert_2.\] By Cauchy-Schwarz, it suffices to show that 
\[ \E_{y,y' \in (\Z/q\Z)^4} f_2(y) \overline{f_2(y')} \E_{x \in (\Z/q\Z)^4}  e_q(rx^T b (y - y')) \leq (q, \det b)^4 q^{-4}\Vert f_2 \Vert_2^2.\] By orthogonality, and since $(r,q) = 1$, the left-hand side is
\begin{align*} \E_{y, y' \in (\Z/q\Z)^4} f_2(y) & \overline{f_2(y')} 1_{b(y - y') \equiv 0 \mdsub{q}} \\ & = q^{-4}\sum_{\substack{h \in (\Z/q\Z)^4 \\ bh \equiv 0 \mdsub{q}}} \E_{y \in (\Z/q\Z)^4} f_2(y) \overline{f_2(y+h)} .\end{align*}
By Cauchy-Schwarz this is at most \[ q^{-4}\# \{ h \in (\Z/q\Z)^4 : bh \equiv 0 \md{q}\} \Vert f_2 \Vert_2^2,\]
and so it is enough to show that 
\begin{equation}\label{suffices-1} \# \{ h \in (\Z/q\Z)^4 : bh \equiv 0 \md{q}\} \leq (q, \det b)^4.\end{equation}
Now if $bh \equiv 0 \md{q}$ then, multiplying on the left by $\adj(b)$, we have $(\det b) h \equiv 0 \md{q}$, i.e. if $h = (h_1,h_2, h_3 ,h_4)$ then $q | (\det b) h_i$ . The number of choices of each $h_i$ is therefore $(q, \det b)$ and so \eqref{suffices-1} follows. This concludes the proof of \eqref{triv-T}.
\end{proof}

The bound in Lemma \ref{triv-bd} is best possible, at least when $(q, \det b) = 1$. To see this, let $\psi : (\Z/q\Z)^4 \rightarrow\C$ be any function with $\Vert \psi \Vert_2 = 1$, and take
\[ f_1(x) := q^2 \E_{y \in (\Z/q\Z)^4} \overline{\psi}(y) e_q(-rQ(x,y)), \quad f_2(x) := \psi(x) .\]
Then one may check using the orthogonality relations that 
\begin{equation}\label{sharp} |T_{f_1,f_2}(r)| = \Vert f_1 \Vert_2 = \Vert f_2 \Vert_2 = 1.\end{equation} A more conceptual explanation of this is as follows. First note that
\begin{equation}\label{t-unit} T_{f_1, f_2}(r) = \overline{\langle \overline{f}_1, \Phi f_2\rangle},\end{equation} where the map $\Phi : \ell^2((\Z/q\Z)^4) \rightarrow \ell^2((\Z/q\Z)^4)$ is given by \[ \Phi f(x) := q^2 \E_{y \in (\Z/q\Z)^4} f(y) e_q(r Q(x, y) ).\] One may then make the key observation that $\Phi$ is unitary (being a composition of invertible dilations, quadratic modulations and Fourier transform). Then we have $f_1 := \overline{\Phi \psi}$, $f_2 = \psi$, and the relations \eqref{sharp} are clear from \eqref{t-unit} and the unitary nature of $\Phi$. 

This also allows a very short (albeit ultimately equivalent) proof of Lemma \ref{triv-bd} in the case $(q, \det b) = 1$. Indeed, by Cauchy-Schwarz and unitarity we have
\[ |T_{f_1, f_2}(r)| = |\langle \overline{f}_1, \Phi f_2 \rangle| \leq \Vert f_1 \Vert_2 \Vert \Phi f_2 \Vert_2 = \Vert f_1 \Vert_2 \Vert f_2 \Vert_2.\]
\vspace*{6pt}
 
One may, using arguments similar to those below, use the bound obtained in Lemma \ref{triv-bd} to show that (roughly speaking)
\[ \int_{\theta \in \bigcup_{r \in (\Z/q\Z)^*} I_{r,q}} |S_{F_1, F_2}(\theta)|d\theta  \lessapprox X^6 q^{-1}  .\] Unfortunately the sum over $q$ does not converge and so this (just) fails to give the desired estimate Proposition \ref{minor-arcs-m2}. It is this, and the sharpness of Lemma \ref{triv-bd}, which ultimately explain the failure of the classical circle method to handle the problem of
quadratic forms in 8 prime variables.

To get around this issue we introduce the following improvement on \eqref{triv-bd} when an average over $r$ is included (at least when $q$ is squarefree and has no small prime factors, and $Q$ is generic).

\begin{proposition}\label{zq-main-prop} There is an absolute constant $\delta > 0$ with the following property. Suppose that $Q$ is generic. Then there is $p_0(Q)$ such that if $q$ is squarefree and with all prime factors greater than $p_0(Q)$, then we have
 \[ \frac{1}{\phi(q)}\sum_{r \in (\Z/q\Z)^*} |T_{f_1, f_2}(r)| \ll q^{- \delta} \Vert f_1 \Vert_2 \Vert f_2 \Vert_2\] for any $f_1, f_2 \in \ell^2((\Z/q\Z)^4)$. 
\end{proposition}
The proof of this proposition occupies most of the rest of the paper. The remainder of this section is devoted to deriving Proposition \ref{minor-arcs-m2} from it.

First we observe that Lemma \ref{triv-bd} and Proposition \ref{zq-main-prop} have a fairly straightforward application to the sums $S_{F_1, F_2}(\theta)$ for $\theta = \frac{r}{q}$, which we record now.

\begin{corollary}\label{saq-cors}
Suppose that $F_1, F_2 : [X]^4 \rightarrow \C$ are $1$-bounded. Suppose that $\det b \neq 0$. Then we have the pointwise bound 
\begin{equation}\label{pointwise-s} |S_{F_1, F_2}(\frac{r}{q})| \ll X^8 q^{-2}.\end{equation}  
Suppose additionally that $q$ is squarefree and has no prime factors of size $\leq p_0(Q)$, and that $Q$ is generic. Suppose that $q \leq KX$. Then 
\begin{equation}\label{average-s} \sum_{r \in (\Z/q\Z)^*} |S_{F_1, F_2}(\frac{r}{q})| \ll  X^8 q^{-1 - \delta}. \end{equation}
\end{corollary}
\begin{proof}
Extend $F_1, F_2$ to functions on $\Z$ by defining them to be $0$ outside of $[X]$. Fix $x_0, y_0 \in q\Z^4$. Then
\begin{align*} \sum_{\substack{x \in x_0 + [q]^4\\ y \in y_0 + [q]^4}} & F_1(x) F_2(y) e_q(rQ(x,y))  \\ & = \sum_{x', y' \in [q]^4} F_1(x_0 + x') F_2(y_0 + y') e_q(r Q(x_0 + x', y_0 + y')).
\end{align*}
Since $x_0, y_0$ are both multiples of $q$ we have $Q(x_0 + x', y_0 + y') \equiv Q(x', y') \md{q}$ and therefore, recalling the definition \eqref{t-def},
 \begin{equation}\label{sum-t-interpret} \sum_{\substack{x \in x_0 + [q]^4\\ y \in y_0 + [q]^4}} F_1(x) F_2(y) e_q(rQ(x,y)) =  q^6  T_{f_1, f_2}(r)\end{equation}
where
 \[ f_1(x') := F_1(x_0 + x')  \quad \mbox{and} \quad f_2(y') := F_2(y_0 + y') .\] 
By Lemma \ref{triv-bd} we therefore have
 \[ \big| \sum_{\substack{x \in x_0 + [q]^4\\ y \in y_0 + [q]^4}} F_1(x) F_2(y) e_q(rQ(x,y)) \big| \leq q^6 (\det b)^2.\]
Covering the range $[X]^4 \times [X]^4$ by $\leq (\frac{X}{q} + 1)^8 \ll X^8 q^{-8}$ boxes of the form $(x_0 + [q]^4) \times (y_0 + [q]^4)$ gives \eqref{pointwise-s}.
 
To obtain \eqref{average-s}, we instead apply Proposition \ref{zq-main-prop} to \eqref{sum-t-interpret}, obtaining
\[ \sum_{r \in (\Z/q\Z)^*} | \sum_{x \in x_0 + [q]^4, y \in y_0 + [q]^4} F_1(x) F_2(y) e_q(rQ(x,y))| \ll q^{6 - \delta} \phi(q) \leq q^{7 - \delta}.\] Dividing into $O(X^8 q^{-8})$ boxes as before gives \eqref{average-s}.
\end{proof}

\begin{proof}[Proof of Proposition \ref{minor-arcs-m2}, assuming Proposition \ref{zq-main-prop}] Recall that 
\[ \mathfrak{m}_2 = \bigcup_{\substack{M' \leq q \leq KX \\ r \in (\Z/q\Z)^*}} I_{r,q},\] with $I_{r,q}$ as defined in \eqref{irq-def}, $M' = \log^{C_2} X$ (with $C_2$ as described in \eqref{q-def}), and $K = 8 \max_{ij} |b_{ij}|$ being a constant associated to the form $Q$. 

Therefore the bound we are trying to prove is
\begin{equation}\label{m2-rpt} \sum_{M' < q \leq KX} \sum_{r \in (\Z/q\Z)^*} \int_{|\eta| \leq M/X^2} |S_{F_1, F_2}(\frac{r}{q} + \eta)| d\eta \ll_A X^6 \log^{-A - 8} X.\end{equation}We begin by using some Fourier analysis to handle the inner integral over $\eta$. 
Recall the definition \eqref{sf1f2-def} of $S_{F_1, F_2}$, that is to say
\[ S_{F_1, F_2}(\theta) := \sum_{x, y \in [X]^4} F_1(x) F_2(y) e(\theta Q(x, y)).\]
Let $w : \R^2 \rightarrow \R$ be some fixed smooth compactly-supported function with $w(u,v) = 1$ for $(u,v) \in [0,1]^2$ and set, for any real parameter $\lambda$,
\begin{equation}\label{W-def} W_{\lambda}(u,v) := w(u,v) e(\lambda Q(u,v)).\end{equation}
Then
\begin{equation}\label{to-fourier} S_{F_1, F_2}(\frac{r}{q} + \eta) = \E_{x,y \in [X]^4} F_1(x) F_2(y) e_q(r Q(x,y)) W_{\eta X^2} (\frac{x}{X}, \frac{y}{X} ).\end{equation}
Now by integration by parts and Leibniz's rule we have
\[ |\hat{W}_{\lambda}(\xi, \xi')| \ll |\xi|^{-2} |\xi'|^{-2} \Vert \frac{\partial^4 W_{\eta}}{\partial^2 u \partial^2 v} \Vert_1 \ll \max(1, |\lambda|^4) |\xi|^{-2} |\xi'|^{-2} .\]
Since we also have the trivial bound 
\[ |\hat{W}_{\lambda}(\xi, \xi')| \leq \Vert W_{\lambda} \Vert_1 = \Vert w \Vert_1 \ll  1\] it follows that 
\begin{equation}\label{W-ell1} \Vert \hat{W}_{\lambda} \Vert_1 \ll \max(1, |\lambda|)^4.\end{equation}
By Fourier inversion
\[ W_{\lambda}(u,v) = \int_{\R} \hat{W}_{\lambda}(\xi, \xi') e(\xi u + \xi' v) d\xi d\xi';\] substituting into \eqref{to-fourier} gives
\[ S_{F_1, F_2}(\frac{r}{q} + \eta) = \int \hat{W}_{\eta X^2}(\xi, \xi') S_{F_{1,\xi}, F_{2,\xi'}} (\frac{r}{q}) d\xi d\xi',\] where
 \[ F_{1,\xi}(x) := F_1(x) e(\xi x/X), \quad F_{2,\xi'}(y) := F_2(y) e(\xi' y/X).\]
 Therefore
 \begin{align}\nonumber\sum_{M' < q \leq KX}& \sum_{r \in (\Z/q\Z)^*} |S_{F_1, F_2}(\frac{r}{q} + \eta)|  \leq \\ & \int_{\R} | \hat{W}_{\eta X^2}(\xi, \xi') | \sum_{M' < q \leq KX}\sum_{r \in (\Z/q\Z)^*} |S_{F_{1,\xi}, F_{2,\xi'}} (\frac{r}{q})| d\xi d\xi'.\label{eq37}\end{align}
We claim the estimate
\begin{equation}\label{m2-enough} \sum_{M' < q \leq KX} \sum_{r \in (\Z/q\Z)^*} |S_{F'_1, F'_2}(\frac{r}{q})| \ll X^8 (M')^{-\delta/100},\end{equation} uniformly for all $1$-bounded $F'_1, F'_2$, where $\delta$ is the exponent appearing in Proposition \ref{zq-main-prop}. Assuming this claim, \eqref{W-ell1} and \eqref{eq37} then imply that
\[ \sum_{M' < q \leq KX}\sum_{r \in (\Z/q\Z)^*} |S_{F_1, F_2}(\frac{r}{q} + \eta)| \ll X^{8} (M')^{-\delta/100} \max(1, |\eta| X^2)^4 .\]  Then, integrating over $|\eta| \leq M/X^2$ we obtain
\[ \int_{\mathfrak{m}_2} |S_{F_1, F_2}(\theta)|d\theta \ll  X^6 M^5(M')^{-\delta/100}  \ll_A X^6 (\log X)^{-A-8}.\] 
For the last step, we recall that we chose $M = \log^{C_1} X$, $M' = \log^{C_2} X$ with the particular choice of $C_1, C_2$ specified at the start of Section \ref{sec3}. This completes the proof of Proposition \ref{minor-arcs-m2}, assuming the claim \eqref{m2-enough}.

Now we must establish \eqref{m2-enough}. The pointwise bound \eqref{pointwise-s} is not good enough, but we do have the improved average bound \eqref{average-s}, albeit only for squarefree $q$ with no small prime factors. Most $q$ do not have this form, and so we need the following lemma to allow us to reduce matters to the consideration of to those that do.

\begin{lemma}\label{factor-lemma}
Suppose that $q = q_0 q_1$ with $(q_0, q_1) = 1$. Then 
\[ \sup_{F_1, F_2} \sum_{r \in (\Z/q\Z)^*} |S_{F_1, F_2}(\frac{r}{q})| \ll q_1^{9} \sup_{F_1, F_2} \sum_{r_0 \in (\Z/q_0 \Z)^*} |S_{F_1, F_2}(\frac{r_0}{q_0})|,\] where in both cases the $\sup$ is over $1$-bounded functions $F_1, F_2 : [X]^4 \rightarrow \C$.
\end{lemma}
\begin{proof}
By the Chinese remainder theorem we have 
\begin{equation}\label{chinese-1} \sum_{r \in (\Z/q\Z)^*} |S_{F_1, F_2}(\frac{r}{q})|  = \sum_{\substack{r_0 \in (\Z/q_0\Z)^* \\ r_1 \in (\Z/q_1\Z)^*}} |S_{F_1, F_2}(\frac{r_0}{q_0} + \frac{r_1}{q_1})|. \end{equation}
Foliating into residue classes modulo $q_1$, we have
\begin{align*} S_{F_1, F_2}(& \frac{r_0}{q_0}  + \frac{r_1}{q_1}) \\ & = \sum_{u,v \in (\Z/q_1\Z)^4} e_{q_1}(r_1 Q(u,v)) \sum_{x,y \in [X]^4} F_{1,u}(x)F_{2,v}(y)  e_{q_0}(r_0 Q(x,y))\end{align*}
where $F_{1,u}(x) := F_1(x) 1_{x \equiv u \mdsub{q_1}}$, $F_{2,v}(y) := F_2(y) 1_{y \equiv v \mdsub{q_1}}$ and so by \eqref{chinese-1}
\begin{align*} \sum_{r \in (\Z/q\Z)^*} & |S_{F_1, F_2}(\frac{r}{q})| \\ & \leq q_1^9 \sup_{u,v} \sum_{r_0 \in (\Z/q_0\Z)^*} \big| \sum_{x,y \in [X]^4} F_{1,u}(x) F_{2,v}(y) e_{q_0}(r_0 Q(x,y))\big| \\ & = q_1^9 \sup_{u,v} \sum_{r_0 \in (\Z/q_0\Z)^*} |S_{F_{1,u}, F_{2,v}}(\frac{r_0}{q_0})|.\end{align*} The lemma follows.
\end{proof}

Let us turn to the actual proof of \eqref{m2-enough}.  Let $p_0 = p_0(Q)$ be the threshold appearing in Proposition \ref{zq-main-prop}. For any $q$, write $q_0$ for the product of all primes $p > p_0$ which divide $q$ precisely once, and set $q_1 := q/q_0$, thus $q_1$ is the product  of all prime powers $p^j \Vert q$ with $p \leq p_0$ or $j \geq 2$. Note that $q_0, q_1$ are coprime. By Lemma \ref{factor-lemma} and \eqref{pointwise-s}, \eqref{average-s} we have
\[ \sum_{r \in (\Z/q\Z)^*} |S_{F_1, F_2}(\frac{r}{q}) | \ll X^8 \min(q^{-1}, q_1^9 q_0^{-1 - \delta}) \leq X^8 \min(q^{-1}, q_1^{11} q^{-1 - \delta}).\]
It therefore suffices to prove that 
\begin{equation}\label{to-est} \sum_{q > M'} \min(q^{-1}, q_1^{11} q^{-1 - \delta}) \ll (M')^{-\delta/100}.\end{equation}
The contribution from $q$ with $q_1 < q^{\delta/22}$ is acceptable (using the second term in the $\min$). 

If $q_1 \geq q^{\delta/22}$ then suppose $q_1 = \prod_p p^{v_p}$ with $v_p \geq 2$ for $p \geq p_0$. Set $q_2 := \prod_p p^{\lfloor v_p/2\rfloor}$. Then $q_2^2 | q$. Moreover, if $v_p \geq 2$ then $p^{\lfloor v_p/2\rfloor} \geq p^{v_p/3}$, so $q_2 \geq c q^{\delta/66}$ (with $c > 0$ depending only on $p_0$). Thus the contribution of these $q$ to \eqref{to-est} can be bounded by

\[ \sum_{d > c(M')^{\delta/66}} \sum_{\substack{q\ll d^{66/\delta} \\ d^2 | q }} q^{-1} \ll  \sum_{d > c(M')^{\delta/66}} \frac{\log d}{d^2} \ll (M')^{-\delta/100}.\]
This concludes the proof.\end{proof}

\section{Exponential sums as matrix coefficients on \texorpdfstring{$\Sp_8(\Z/q\Z)$}{Sp8(Z/qZ)}}\label{sec7}

The remainder of the paper is occupied with the proof of Proposition \ref{zq-main-prop}, the statement of which the reader may wish to recall at this point.

We remarked after the proof of Lemma \ref{triv-bd} that we can write
\begin{equation}\label{inner-prod-form} T_{f_1,f_2}(r) = \overline{\langle \overline{f}_1, \Phi_r f_2\rangle},\end{equation} where
\[ \Phi_r f = q^2 \E_y f(y) e_q(r Q(x,y))\] is unitary.
The crucial observation which drives our whole argument is that the subgroup of $\U (\ell^2((\Z/q\Z)^4))$ (the group of all unitary operators on $\ell^2((\Z/q\Z)^4)$) generated by operators $\Phi$ of this type (over all $Q$) is rather small. Indeed, as we shall shortly see, it has size $q^{36 + o(1)}$. This means that the specific operators $\Phi_r$ (with $Q$ fixed but $r$ allowed to vary over $(\Z/q\Z)^*$) already occupy a reasonable portion of this group.

This group turns out to be the symplectic group $\Sp_8(\Z/q\Z)$. Let us recall what these groups are, in a very concrete way. Let $R = \Z/q\Z$ with $q$ odd. Then $\Sp_8(R)$ is a group of $8 \times 8$ matrices over $R$, which we will write in $2 \times 2$ block form with each block being a $4 \times 4$ matrix.

\begin{definition}
We define the symplectic group $\Sp_8(R)$ to be the group consisting of all $8 \times 8$ block matrices $g = \smallpmat{A}{B}{C}{D}$ with entries in $R$ and $A^T C = C^T A$, $B^T D = D^T B$ and $A^T D - C^T B = I$.
\end{definition}
Define
\begin{equation}\label{J-def} J := \pmat{0}{I}{-I}{0}.\end{equation}

It is an simple exercise to check that $g \in \Sp_8(R)$ if and only if $g^T J g = J$. In fact, this is the more usual definition of the sympletic group, but it suits us to be more explicit.

Note that if $g = \smallpmat{A}{B}{C}{D} \in \Sp_8(R)$ then $g$ is left-invertible with left-inverse $\smallpmat{D^T}{-B^T}{-C^T}{A^T}$. This is then, of course, also a right-inverse for $g$, and this gives us the additional relations
\[ AB^T = B A^T, \; CD^T = DC^T \quad  \mbox{and} \quad AD^T - BC^T = I\] for any symplectic matrix. 

We will also need the fact that $(\Z/q\Z)^*$ acts on $\Sp_8(\Z/q\Z)$ by ``dilation'' automorphisms. If $g = \smallpmat{A}{B}{C}{D}$ and if $r \in (\Z/q\Z)^*$ then we define
\begin{equation}\label{gr-def} g^{(r)} = \pmat{A}{r^{-1} B}{rC}{D}.\end{equation}  It is then easy to see that this gives an action of $(\Z/q\Z)^*$ on $\Sp_8(\Z/q\Z)$ by automorphisms.

Finally, we note for future reference (see, for example, \cite{neuhauser}) that
\begin{equation}\label{sp-8-size} |\Sp_8(\F_p)| = p^{16} (p^8 - 1)(p^6 - 1) (p^4 - 1)(p^2 - 1)  = (1 + o(1)) p^{36}.\end{equation}

Now we come to the key link between $\Sp_8$ and operators such as $\Phi_r$ in \eqref{inner-prod-form}, which stems from \cite{weil} and is thus known as the Weil representation (or, depending on the context, the Segal-Shale-Weil representation or the oscillator representation).

\begin{proposition}[Weil representation]\label{weil-prop}
Let $q$ be squarefree and odd. Then there is a unitary representation \[ \rho : \Sp_8(\Z/q\Z) \rightarrow \U(\ell^2((\Z/q\Z)^4))\] and a function $\xi : \Sp_8(\Z/q\Z) \rightarrow \{z \in \C : |z| = 1\}$ satisfying the following:
\begin{itemize}
\item \textup{(Dilations)} If $g = s(E)$ where $s(E) := \smallpmat{E}{0}{0}{E^{-T}}$ with $E$ invertible then
\[ \rho(g) f(x) = \xi(s(E)) f(E^{-1} x);\]
\item \textup{(Fourier transform)} If $g = J$ then
\[ \rho(g) f(x) =  \xi(J) q^2 \E_{y \in (\Z/q\Z)^4} f(y) e(x^T y);\]
\item \textup{(Quadratic modulations)} If $g = l(W)$ where $l(W) := \smallpmat{I}{0}{W}{I}$ with $W$ symmetric then 
\[ \rho(g) f(x) = \xi(l(W))e_q(-\frac{1}{2} x^T W x)f(x).\]
\end{itemize}
\end{proposition}
\emph{Remarks.} \emph{1.} There is nothing special about $\Sp_8$ here; similar results hold for $\Sp_{2m}$ for any positive integer $m$. The Weil representation is well-known over $\R$ (where one needs to pass to the double cover of the symplectic group), but in finite situations it seems to me that it is only at all widely discussed over finite fields. In this case, the construction is given in detail in (for example) the paper \cite{neuhauser} of Neuhauser. The analogue of this in the lower-dimensional setting of $\Sp_2(\Z/p\Z) = \SL_2(\Z/p\Z)$ already contains the key ideas, and a very nice description of this may be found in the notes of Charlotte Chan \cite{chan-notes}, which I found helpful in preparing this material. It is not difficult to derive the case $q$ squarefree from the prime case, and we do this in Appendix \ref{weil-app}. 

\emph{2.} The phase $\xi$ can be given explicitly if desired. When $q = p$ is an odd prime, we can take $\xi(l(W)) = \xi(J) = 1$ and $\xi(s(E)) = \left( \frac{\det E}{p}\right)$, and the general squarefree case can then be deduced from the arguments in Appendix \ref{weil-app}. For details of these calculations (which are somewhat involved) see \cite{neuhauser}. In this paper, we will not need explicit values of $\xi$, and the mere existence is a much easier statement to prove, this being \cite[Theorem 4.3]{neuhauser}.

\emph{3.} If desired one can also add in the translations $f(x) \mapsto f(x - v)$ and the linear modulations $f(x) \mapsto e(-t^T x) f(x)$, getting an action by the ``Jacobi group'' $\Sp_8(\Z/q\Z) \ltimes \Heis_8(\Z/q\Z)$, where $\Heis_8$ is the Heisenberg group on $(\Z/q\Z)^8 \times \Z/q\Z$.

\emph{4.} It is not really correct to call $\rho$ ``the'' Weil representation. In the case $q = p$ a prime, further representations $\tilde\rho$ of the same dimension can be obtained by twisting with the dilation $\sigma(g) = g^{(r)}$, that is to say $\tilde\rho(g) := \rho(g^{(r)})$. When $r$ is not a square in $\F_p^*$, the dilation is an outer automorphism and it is known that $\tilde\rho \not\cong \rho$. (In the literature this would be described in terms of different central characters on the Heisenberg group giving different Weil representations, see \cite[Section 7]{neuhauser} or \cite[Proposition 4]{szechtman}). Thus there are two Weil representations of $\Sp_8(\Z/p\Z)$, and the one we are considering is a concrete realisation of one of them.

When $q = p_1 \cdots p_n$, one may obtain $2^n$ non-isomorphic representations of $\Sp_8(\Z/q\Z)$ by taking tensor products. The representation whose existence we assert in Proposition \ref{weil-prop} is one of these. However, it turns out \emph{not} to be simply the tensor product of the $\rho$s associated to each $p_i$: we must first apply some twists. The details are given in Appendix \ref{weil-app}.

\emph{5.} Even when $q = p$, the Weil representation is not irreducible. It splits into its actions on odd and even functions, which \emph{are} irreducible representations of degrees $\frac{1}{2}(p^4 \pm 1)$. We will not need this fact here. \vspace*{8pt}

Now that we have defined the Weil representation, we can interpret the exponential sums $T_{f_1, f_2}(r)$ as matrix coefficients. We begin with an important definition which will be relevant for the rest of the paper.

\begin{definition}[Symplectic element]\label{symplectic-element}
Suppose that $Q(x,y) = x^T a x + x^T b y + y^T c y$ is a quadratic form and that $\det b \neq 0$. Then we associate to $Q$ the element $g = g(Q) \in \Sp_8(\Q)$ defined by
\[ g :=  \pmat{-2 b^{-T} c}{b^{-T}}{4ab^{-T} c - b}{-2a b^{-T}}.\] 
We call this the \emph{symplectic element} associated to $Q$.
\end{definition}

As mentioned in the introduction, we will abuse notation by regarding $g$ as an element of $\Sp_8(\Z/q\Z)$ for squarefree odd $q$, coprime to $\det b$. Here is the promised interpretation of exponential sums as matrix coefficients.

\begin{proposition}\label{quad-form-to-matrix} Suppose that $q$ is odd, squarefree and coprime to $\det b$. Then for any $f_1, f_2 \in \ell^2((\Z/q\Z)^4)$ we have
\[ |T_{f_1, f_2}(r)| = |\langle \overline{f}_1, \rho(g^{(r)}) f_2\rangle|,\]
where $g^{(r)}$ is the dilate of $g$ \textup{(}regarded as an element of $\Sp_8(\Z/q\Z)$\textup{)} by $r$ as defined in \eqref{gr-def} and $\rho$ is the Weil representation described in Proposition \ref{weil-prop}.
\end{proposition}
\begin{proof}
It suffices to establish the case $r = 1$, since then the formula for general $r$ follows by applying that case with $Q$ replaced by $rQ$ (or, to be pedantic, $\overline{r}Q$ for some $r \in \Z$ projecting to $r \md{q}$).

To handle the case $r = 1$, note that $q^2\E_{y \in (\Z/q\Z)^4} f(y) e(Q(x,y))$ may be built up as a composition of four (unitary) operations, as follows:
\begin{enumerate}
\item A quadratic modulation $f(y) \mapsto f(y) e_q(y^T c y)$;
\item Fourier transform $f \mapsto q^2\E_{y \in (\Z/q\Z)^4} f(y) e_q(x^T y)$;
\item Dilation $f(x) \mapsto f(b^T x)$;
\item Quadratic modulation $f(x) \mapsto f(x)e_q(x^T a x)$. 
\end{enumerate}
In the Weil representation these four operations correspond, up to scalar multiplication by unit complex numbers, respectively, to the following elements of $\Sp_8(\Z/q\Z)$: $l(-2c),  J, s(b^{-T})$ and $l(-2a)$. Therefore by Proposition \ref{weil-prop} (since $\rho$ is a homomorphism!) we have
\[ q^2\E_{y \in (\Z/q\Z)^4} f(y) e(Q(x,y) ) = z \rho(g) f(x)\] for some unit complex number $z= z(Q)$ where \[ g = l(-2a) \cdot s(b^{-T}) \cdot J \cdot l(-2c)\] is the product of the four elements just written down. A short computation confirms that $g$ is the symplectic element of $Q$ as defined in Definition \ref{symplectic-element}. 

Finally, we have
\begin{align*} T_{f_1, f_2}(1) & = \E_{x \in (\Z/q\Z)^4} f_1(x) \E_{y \in (\Z/q\Z)^4} f_2(y) e(Q(x,y)) \\ & = z \E_{x \in (\Z/q\Z)^4} f_1(x) (\rho(g) f_2)(x) \\ & = z \overline{ \langle \overline{f}_1, \rho(g) f_2\rangle   }.\end{align*}
This completes the proof.\end{proof}

The following definition will play a key role in what follows.

\begin{definition}\label{measure-def}
Fix $Q(x,y) = x^T a x + x^T b y + y^T c y$, a quadratic form over $\Z$ with $\det b \neq 0$. Let $g$ be the symplectic element of $Q$ (see Definition \ref{symplectic-element}). Then for every odd squarefree $q$ coprime to $\det b$ we associate a probability measure $\mu_q$ on $\Sp_8(\Z/q\Z)$, which puts weight $\frac{1}{\phi(q)}$ on each of the points $g^{(r)}$, $r \in (\Z/q\Z)^*$. \end{definition}

We are now in a position to rephrase Proposition \ref{zq-main-prop} in terms of matrix coefficients.

\begin{proposition}\label{matrix-form}
Suppose that $Q$ is generic. Then there is $p_0(Q)$ such that the following is true. Suppose that $q$ is squarefree with all prime factors greater than $p_0(Q)$. Then
\[ \int  |\langle f_1, \rho(x) f_2\rangle| d\mu_q(x) \ll q^{-\delta} \Vert f_1 \Vert_2 \Vert f_2 \Vert_2\]  for all $f_1, f_2 \in \ell^2((\Z/q\Z)^4)$.
\end{proposition}
\emph{Remarks.} For a discussion of the integral notation, see the start of the next section. By taking $p_0(Q)$ larger than any prime factor of $\det b$, we can ensure that the conditions of Proposition \ref{quad-form-to-matrix} are satisfied. For notational simplicity we switched $\overline{f}_1$ to $f_1$, which makes no difference since these functions have the same $\ell^2$-norm.

\section{Averages of matrix coefficients}\label{sec8}

In this section we give a general bound for averages of matrix coefficients. Whilst we do not know of a reference for quite this result, the first part of the argument is related to earlier work, particularly Bourgain \cite{bourgain-modular-hyperbola} and Skhredov \cite[Section 4]{shkredov}. The second idea, of using quasirandomness (no small-dimensional representations apart from the trivial one), is also by no means new. It is exploited in related ways in many works starting with Sarnak--Xue \cite{sarnak-xue} and continuing with, for instance, Bourgain--Gamburd \cite{bourgain-gamburd} and Gowers \cite{gowers-quasirandom}. \vspace*{6pt}

\emph{Probability measures.} We begin by recalling some basic notions about probability measures on (finite) groups. Let $G$ be a finite group. A probability measure $\mu$ on $G$ is simply a function $\mu : G \rightarrow [0,1]$ with $\sum_{x \in G} \mu(x) = 1$. The opposite measure $\mu^{\circ}$ is defined by $\mu^{\circ}(x) = \mu(x^{-1})$. If $\mu^{\circ} = \mu$ then we say that $\mu$ is \emph{symmetric}. If $\mu_1, \mu_2$ are two probability measures then their convolution $\mu_1 \ast \mu_2$ is defined by $\mu_1 \ast \mu_2(x) = \sum_{g_1g_2 = x} \mu_1(g_1) \mu_2(g_2)$. This is also a probability measure. If $\mu$ is a probability measure and $2m$ a positive even integer, we write $\mu^{(2m)}$ for the $2m$-fold \emph{symmetrised} convolution power $\mu^{\circ} \ast \mu \ast \mu^{\circ} \ast \cdots \ast \mu$. This is slightly non-standard, but very convenient as these are the only types of convolution power we will consider in this paper. At one place in Section \ref{unif-conv-power} we will use a similar notation with an odd power, thus $\mu^{(2m-1)}(x) = \mu^{\circ} \ast \mu \ast \mu^{\circ} \ast \cdots \ast \mu^{\circ}$. This, of course, is not necessarily a symmetric measure. We have $\mu^{(2m)} = \mu^{(2m - 1)} \ast \mu$. If $H \leq G$ is a subgroup then we write $\mu_H$ for the uniform probability measure on $H$, that is to say $\mu_H(x) = |H|^{-1} 1_{x \in H}$. 

If $\mu$ is a probability measure on a finite group $G$ then we write $\Vert \mu \Vert = \big(\sum_x \mu(x)^2\big)^{1/2}$. Note that this is normalised differently to the $\ell^2$-norm of functions which has appeared in previous sections: to reduce the potential for confusion, we omit any subscript from the norm. We extend this notion to differences of measures in the obvious way, thus $\Vert \mu - \nu \Vert = \big(\sum_x (\mu(x) - \nu(x))^2 \big)^{1/2}$.

If $F : G \rightarrow \C$, we will adopt the fairly standard convention in this context of writing $\int F(x) d\mu(x)$ instead of $\sum_x F(x)\mu(x)$.\vspace*{6pt}
 
We will need the following consequence of Schur's lemma which is standard but cannot be reliably found in every textbook.

\begin{lemma}\label{schur}
Let $\psi : G \rightarrow \U(V)$ be an irreducible representation of a finite group $G$. Suppose that $v, w \in V$. Then we have
\[ \int |\langle v, \psi(x) w \rangle|^2 d\mu_G(x) = \frac{1}{\dim \psi} \Vert v \Vert^2 \Vert w \Vert^2.\]
\end{lemma}
\begin{proof} See \cite[Proposition 4.3.5]{kowalski}. Note that here $\dim \psi$ is defined to be $\dim V$.\end{proof}
 
\begin{proposition}\label{matrix-avg}
Let $G$ be a finite group, and let $\rho : G \rightarrow \U(V)$ be a finite-dimensional unitary representation of $G$. Let $\mu$ be a probability measure on $G$. Let $H \leq G$ be the group generated by $\Supp(\mu^{(2)}) = \Supp(\mu^{\circ} \ast \mu)$. Suppose that
\begin{itemize}
\item \textup{(Almost uniform distribution of convolution powers)} For some real number $K \geq 1$ and for some power of two $m$ we have \begin{equation}\label{almost-unif} \mu^{(m)}(x) \leq K \mu_H(x)\end{equation} pointwise;
\item \textup{(Quasirandomness)} If $\rho|_H = \bigoplus_i \psi_i$ as a sum of irreducible representations \textup{(}of $H$\textup{)} then $\dim \psi_i \geq D$ for all $i$. 
\end{itemize}
Then we have the bound
\begin{equation}\label{int-star} \int  |\langle v, \rho(x) w\rangle| d\mu(x) \leq K^{1/m} D^{-1/2m}\end{equation} for all $v, w \in V$ with $\Vert v \Vert = \Vert w \Vert = 1$.
\end{proposition}
\emph{Remark.} Note that the trivial bound is $1$ (by the unitary nature of $\rho$ and Cauchy-Schwarz). If $K \approx 1$ and $D$ is somewhat large, \eqref{int-star} is therefore an appreciable improvement of the trivial bound.
\begin{proof}
Set 
\[ \eta := \int  |\langle v, \rho(x) w\rangle| d\mu(x).\]
For each $x \in G$, let $\xi(x) = e(\arg \langle v, \rho(x)w\rangle)$, so $\xi(x)$ is a unit complex number and
\[ \langle v, \int \xi(x) \rho(x)w d\mu(x) \rangle =  \int \overline{\xi(x)} \langle v, \rho(x) w \rangle d \mu(x) = \eta.\] By Cauchy-Schwarz, 
\[ \big\Vert \int \xi(x) \rho(x) wd\mu(x)  \big\Vert \geq \eta.\] Squaring and expanding out gives
\[ \int \xi(x) \overline{\xi(x')} \langle  \rho(x) w, \rho(x') w\rangle d\mu(x) d\mu(x')  \geq \eta^2,\] thus
\[ \int |\langle \rho(x) w, \rho(x') w\rangle | d\mu(x) d\mu(x') \geq \eta^2.\] Since $\rho$ is a unitary representation, this implies
\[ \int|\langle w, \rho(x^{-1} x') w\rangle| d\mu(x) d\mu(x') \geq \eta^2,\]
or in other words
\[\int |\langle w, \rho(x) w\rangle| d\mu^{(2)}(x) \geq \eta^2.\]
We may now apply the same argument again repeatedly, noting that $\mu^{(2)}, \mu^{(4)},\dots$ are symmetric, to obtain
\[ \int |\langle w, \rho(x) w \rangle |d\mu^{(m)}(x) \geq \eta^{m}\] for any power of two $m$. 
By the almost uniform distribution assumption \eqref{almost-unif}, this implies (with $m$ as in \eqref{almost-unif}) that
\begin{equation}\label{rep-lower} \int |\langle w, \rho(x) w \rangle| d\mu_H(x) \geq K^{-1} \eta^{m}.\end{equation}
Now decompose $V = \bigoplus_{i=1}^n V_i$ as a sum of orthogonal $\rho(H)$-invariant subspaces, irreducible for $\rho|_H$. Let $w_i$ be the projection of $w$ to $V_i$, so $w = \sum_{i=1}^n w_i$ and 
\begin{equation}\label{sum-squares} \sum_{i=1}^n \Vert w_i \Vert^2 = 1.\end{equation} 
By Lemma \ref{schur} we have for $i = 1,\dots, n$
\[ \int |\langle w_i, \rho(x) w_i\rangle|^2d\mu_H(x)  = \frac{1}{\dim V_i} \Vert w_i \Vert^4,\] so by Cauchy-Schwarz and the quasirandomness assumption 
\begin{equation}\label{schur-quasi} \int |\langle w_i, \rho(x) w_i\rangle | d\mu_H(x) \leq D^{-1/2} \Vert w_i \Vert^2.\end{equation}
Since (by orthogonality) 
\[ \langle w, \rho(x) w\rangle = \sum_{i=1}^n \langle w_i, \rho(x) w_i\rangle\] for all $x \in H$, it follows from \eqref{sum-squares} and \eqref{schur-quasi} that 
\[ \int |\langle w, \rho(x) w \rangle |d\mu_H(x) \leq D^{-1/2}.\]
Comparing this with \eqref{rep-lower} gives the claimed bound.
\end{proof}

We now outline the rest of the paper. Recall that we have reduced the proof of our main theorem to the task of proving Proposition \ref{matrix-form}. We now have a tool, Proposition \ref{matrix-avg}, to use on this problem.  However, we must verify the two requirements, the uniform distribution property \eqref{almost-unif} and the quasirandomness property, in our setting. 

The formal statements are Propositions \ref{subgroup-spread} and \ref{quasirandom-gamma} below. First, we give a definition which will play an important role in the rest of the paper.

\begin{definition}\label{gamma-def} Suppose that $q$ is odd, squarefree and coprime to $\det b$. Let $\Gamma_q \leq \Sp_8(\Z/q\Z)$ be the group generated by the elements $g^{-(r)} g^{(s)}$, $r, s \in (\Z/q\Z)^*$, or equivalently by the support of $\Supp(\mu_q^{(2)})$. \end{definition}

\begin{proposition}\label{subgroup-spread}
Let $Q$ be a quadratic form. Then there is some $p_0(Q)$ such that the following is true. Suppose that $q$ is squarefree and has all prime factors greater than $p_0(Q)$. Let $\mu_q$ be the measure described in Definition \ref{measure-def}. Then there is power of two $m = O(1)$ such that $\mu_q^{(m)} \ll \mu_{\Gamma_q}$ pointwise, with the implied constant being absolute.
\end{proposition}

We will prove this in Section \ref{unif-conv-power}. It does not require any genericity assumption on the form $Q$. 

\begin{proposition}\label{quasirandom-gamma}
Let $Q$ be a generic quadratic form. Then there is some $p_0(Q)$ such that the following is true. Let $\rho$ be the Weil representation on $\Sp_8(\Z/q\Z)$ \textup{(}as given in Proposition \ref{weil-prop}\textup{)}. Suppose that $q$ is squarefree and has all prime factors greater than $p_0(Q)$. Then $\rho|_{\Gamma_q}$ splits into irreducible subrepresentations of dimensions $\geq q^{1 - o(1)}$. 
\end{proposition}

We will prove this in Sections \ref{sec10} and \ref{sec11}.

Propositions \ref{subgroup-spread} and \ref{quasirandom-gamma} are precisely what is needed in order to apply Proposition \ref{matrix-avg}, and the conclusion is precisely Proposition \ref{matrix-form}. Therefore we have, as the remaining outstanding tasks, the proofs of these two propositions.

\section{Uniform distribution of convolution powers}\label{unif-conv-power}

In this section we establish Proposition \ref{subgroup-spread}. Our proof of this statement has a hint of model theory about it, though we will not use that language. As remarked in the introduction, it is somewhat related to Tao's argument in \cite{tao-spectral}. Here is a rough plan of the proof. 
\begin{enumerate}
\item (Step 1) Consider first the case $q = p$ a sufficiently large prime. We argue that the sequence $\Vert \mu_p^{(2^j)} \Vert$, $j = 0,1,2,\dots$ (which is non-increasing by Young's inequality) stabilises at some time $t = O(1)$, in the sense that $\Vert \mu_p^{(2^{t+1})} \Vert = (1 + O(p^{-1/2})) \Vert \mu_p^{(2^t)} \Vert$. This uses the Lang-Weil bound from algebraic geometry. 
\item (Step 2) By standard arguments from additive combinatorics (recalled in Appendix \ref{appc}), this implies that $\Vert \mu_p^{(2^t)} - \mu_H \Vert \ll p^{-c} \Vert \mu_H \Vert$, for some subgroup $H \leq \Sp_8(\Z/p\Z)$.
\item (Step 3) By some group-theoretic arguments, $H$ must in fact be $\Gamma_p$.
\item (Step 4) Taking a few further convolution powers, we upgrade the estimate to a much stronger bound $\Vert \mu^{(s)} - \mu_{\Gamma_p} \Vert_{\infty} \ll p^{-100}$.
\item (Step 5) We deduce the general squarefree case of Proposition \ref{subgroup-spread}.
\end{enumerate}

\emph{Lang-Weil estimate.} We keep algebro-geometric terminology to an absolute minimum. A good down-to-earth account of what we need may be found in \cite[Chapter 9]{tao-expansion}. Let $M$ be a real parameter. Then for the purposes of this paper, by a variety of complexity $\leq M$ defined over $\F_p$ we mean a set of points of the form 
\[ V := \{x \in  \overline{\F}_p^n : P_1(x) = \cdots = P_m(x) = 0\},\] where $P_1,\dots, P_m \in \F_p[X_1,\dots, X_n]$ are polynomials all of degree $\leq M$, and $m,n \leq M$. Denote by $V(\F_p)$ the $\F_p$-points of $V$, that is to say the points of $V$ all of whose coordinates lie in $\F_p$.
\begin{proposition}[Lang-Weil]\label{lang-weil}
We have
\[ |V(\F_p)| = (c(V) + O_M(p^{-1/2})) p^{\dim V},\] for some integer $c(V) = O_M(1)$.
\end{proposition}
\emph{Remark.} In fact, $c(V)$ is the number of top-dimensional components of $V$ which are definable over $\F_p$, but we shall not need this description. Nor will we need to really know what dimension means, other than that it is an integer in the range $0 \leq \dim V \leq n$. All we need is the fact that the quantities $|V(\F_p)|$ are restricted to a rather discretised set of values. This kind of application of Lang-Weil has appeared in several model-theoretic works and is related to the concept of \emph{stability}.

The Lang-Weil estimate has the following consequence for convolution powers of our measures $\mu_p$. 
\begin{lemma}\label{lem9.2}
There are functions $\alpha, \beta : \N \times \{\textup{primes}\} \rightarrow \Z_{\geq 0}$ and non-decreasing functions $\alpha_* \rightarrow \Z_{\geq 0}$ and $p_0 : \N \rightarrow \N$ such that for all $j \geq 1$ and for all primes $p$ we have
\begin{equation}\label{disc-ell-2} \Vert \mu_p^{(2^j)} \Vert^2 = (\alpha(j,p) + O_j(p^{-1/2})) p^{-\beta(j,p)},\end{equation} where $\alpha(j,p), \beta(j,p)$ are integers with $0 < \alpha(j,p) \leq \alpha_*(j)$.
Moreover, if $p \geq p_0(j)$ then we have
\begin{equation}\label{b-bounds} 0 \leq \beta(j,p) \leq 36;\end{equation}
\begin{equation}\label{dimension-reduct} \beta(j+1,p) \geq \beta(j,p)\end{equation} and
\begin{equation}\label{component-reduct} \alpha(j+1,p) \leq \alpha(j,p) \quad \mbox{if} \quad \beta(j+1,p) = \beta(j,p).\end{equation}
\end{lemma}
\begin{proof}
The key point is to interpret $\Vert \mu_p^{(2^j)} \Vert^2$ in terms of the number of $\F_p$-points on a variety of bounded complexity. To this end, we have
\begin{align*} &  \Vert \mu^{(2^j)}_p \Vert^2  = \sum_{x} \mu_p^{(2^j)}(x)^2 \\ & = (p-1)^{-2^{j+1}}\# \{ (r, r') \in (\F_p^*)^{2^{j}} \times (\F_p^*)^{2^{j}}  : g^{-(r_1)} g^{(r_2)} \cdots g^{(r_{2^j})} = \\ & \qquad\qquad\qquad \qquad\qquad\qquad \qquad\qquad\qquad = g^{-(r'_1)} g^{(r'_2)} \cdots  g^{(r'_{2^j})}\} \\ & =  (p-1)^{-2^{j+1}}\# \{ (r,r', y, y') \in \F_p^{2^{j}} \times \F_p^{2^j} \times \F_p \times \F_p : r_1 \cdots r_{2^j} y = \\  & \qquad\qquad = r'_1 \cdots r'_{2^j} y' = 1, \; \;   g^{-(r_1)} g^{(r_2)} \cdots g^{(r_{2^j})} =  g^{-(r'_1)} g^{(r'_2)} \cdots  g^{(r'_{2^j})}\} \\ & = (p-1)^{-2^{j+1}} |V_j(\F_p)| = (1 + O_j(\frac{1}{p})) p^{-2^{j+1}} | V_j(\F_p)|,\end{align*} where $V_j \subset \overline{\F}_p^{2^{j+1} + 2}$ is some variety of complexity $O_j(1)$, defined over $\F_p$. Note here that, although (for instance) $g^{(r_1)} = \smallpmat{-2b^{-T} c}{r_1^{-1} b^{-T}}{r_1(4ab^{-T}c - b)}{-2ab^{-T}}$ is not \emph{a priori} given by polynomials, we can express
\[ g^{(r_1)} = \pmat{-2b^{-T} c}{r_2 \cdots r_{2^j} y b^{-T}}{r_1 (4ab^{-T} c - b)}{-2ab^{-T}},\] and this \emph{is} given by polynomials. (Alternatively, one could talk about quasiprojective varieties, but the trick of introducing $y, y'$ avoids the need to do that.)

This immediately implies, by the Lang-Weil bound, the first statement \eqref{disc-ell-2} (with $\beta(j,p) = 2^{j+1} - \dim V_j$). We now proceed to derive the additional statements \eqref{b-bounds}, \eqref{dimension-reduct} and \eqref{component-reduct}, which we do by combinatorial means (with reference to \eqref{disc-ell-2}).

For \eqref{b-bounds}, note that \emph{any} probability measure $\nu$ on a finite group $G$ satisfies $|G|^{-1} \leq \Vert \nu \Vert^2 \leq 1$; the lower bound is Cauchy-Schwarz, and the upper bound is the trivial bound (with equality only if $\nu$ is concentrated at one point). Since (see \eqref{sp-8-size}) $|\Sp_8(\Z/p\Z)| = (1 + o(1)) p^{36}$, \eqref{b-bounds} follows if $p$ is large enough.

For items \eqref{dimension-reduct} and \eqref{component-reduct} we use Young's inequality (Lemma \ref{young}), which implies that $\Vert \mu_p^{(2^j)} \Vert^2$ is a non-increasing function of $j$. Therefore we have, by \eqref{disc-ell-2}, 
\[ (\alpha(j+1,p) + O_j(p^{-1/2})) p^{-\beta(j+1,p)} \leq (\alpha(j,p) + O_j(p^{-1/2})) p^{-\beta(j,p)}.\]
Here, $\alpha(j+1,p), \alpha(j,p)$ are positive integers of size $O_j(1)$ and so by taking $p$ sufficiently large in terms of $j$ both \eqref{dimension-reduct} and \eqref{component-reduct} follow.

Finally, note that $\alpha(j,p)$ is bounded above by $O_{M_j}(1)$, where $M_j = O_j(1)$ is an upper bound for the complexity of $V_j$. Thus $\alpha(j,p) \leq \alpha_*(j)$ for some function $\alpha_*$, which we may clearly assume to be non-decreasing (else replace it by $\sup_{i \leq j} \alpha_*(i)$). For the same reasons, we may also assume that $p_0$ is a non-decreasing function.
\end{proof}
\emph{Remark.} There should probably be a more purely algebro-geometric way to see the inequalities \eqref{dimension-reduct}, \eqref{component-reduct} in terms of the number of connected components and dimensions of the relevant $V$s, but this would certainly take much longer to set up than the argument we gave above.

We now turn to the realisation of Step 1 of the outline.

\begin{proposition}\label{stable-ell-2}
Suppose that $p$ is sufficiently large. Then there is some $t = O(1)$ such that $\Vert \mu_p^{(2^{t+1})} \Vert = (1 + O(p^{-1/2}))\Vert \mu_p^{(2^t)} \Vert$.
\end{proposition}
\begin{proof}
We use Lemma \ref{lem9.2} and the notation there. It clearly suffices to show that, for some $t = O(1)$, we have 
\begin{equation}\label{to-show} \alpha(t+1, p) = \alpha(t,p) \quad \mbox{and} \quad \beta(t+1,p) = \beta(t,p).\end{equation}

Define a sequence $T_1,T_2,\dots, T_{37}$ as follows. Set $T_1 := \alpha_*(1)$, and then inductively define $T_{i+1} := T_i + \alpha_*(T_i)$ for $i = 1,2,3,\dots, 36$. If $p \geq p_0(T_{37})$ we have the bounds \eqref{b-bounds}, \eqref{dimension-reduct} and \eqref{component-reduct}, for all $j \leq T_{37}$. 

We claim that there is some $t \leq T_{37}$ such that \eqref{to-show} holds. Suppose not. Then, by \eqref{dimension-reduct}, \eqref{component-reduct} we have that for all $j \leq T_{37}$ either 
\begin{enumerate}
\item  $\beta(j+1,p) > \beta(j,p)$ or
\item  $\beta(j+1, p) = \beta(j,p)$ and $\alpha(j+1,p) < \alpha(j,p)$. 
\end{enumerate}
By \eqref{b-bounds}, there are at most $36$ values of $j$ for which (1) occurs; suppose they are $T'_1,\dots, T'_m$, $m \leq 36$. 

For $j = 1,2,\dots, T'_1 - 1$ we must have (2), which means that $\alpha(T'_1,p) \leq \alpha(1,p) - T'_1\leq \alpha_*(1) - T'_1 +1 = T_1 - T'_1 + 1$. Since $\alpha(j,p)$ is always positive, this implies that $T'_1 \leq T_1$. 

Now for $j = T'_1 + 1, \dots, T'_2 - 1$ we must also have (2), which means that 
\begin{align*}
\alpha(T'_2, p) & \leq \alpha(T'_1, p) - (T'_2 - T'_1 -1)\\ & \leq \alpha_*(T'_1) -  (T'_2 - T'_1 - 1) \leq \alpha_*(T_1) + T_1 - T'_2 + 1.
\end{align*}
Since $\alpha(j,p)$ is always positive, this implies that $T'_2 \leq \alpha_*(T_1) + T_1 = T_2$. Continuing in this manner we see inductively that $T'_m \leq T_m \leq T_{36}$.
Continuing now with $j = T'_m + i$, $i = 1,2,\dots$, only (2) can occur, and so 
\[ \alpha(T'_m + i, p) \leq \alpha(T'_m, p) - i \leq \alpha_*(T'_m) - i \leq \alpha_*(T_{36}) - i.\]
Since (yet again) $\alpha(j,p)$ is always positive, this can only continue as far as $i = \alpha_*(T_{36})-1$ before we get a contradiction. Note that then $j  \leq T'_{m} + \alpha_*(T_{36}) \leq T_{37}$, so all the appeals we made to \eqref{b-bounds}, \eqref{dimension-reduct} and \eqref{component-reduct} were indeed valid. 

This contradiction shows that we were wrong to assume that there is no $t \leq T_{37}$ for which \eqref{to-show} holds.
\end{proof}

\textbf{Step 2.} The conclusion of Step 1 (Proposition \ref{stable-ell-2}) is that for some $t$, $1 \leq t \ll O(1)$, we have, for the symmetic measure $\nu := \mu^{(2^t)}$ the very strong ``flattening''
\[  \Vert \nu \ast \nu \Vert = (1 + O(p^{-1/2})\Vert \nu \Vert .\]
It is well-known that any probability measure satisfying this kind of property is close to uniform on a subgroup. The precise statement we need is Corollary \ref{stability-young-cor} in Appendix \ref{appc}, from which we conclude that there is some subgroup $H$ such that 
\begin{equation}\label{ell-2-close} \Vert \mu_p^{(2^t)} - \mu_H \Vert \ll p^{-c} \Vert \mu_H \Vert,\end{equation}
\begin{equation}\label{mu-power-group} \mu_p^{(2^t)}(H) \geq 1 - O(p^{-c}),\end{equation}
and
\begin{equation}\label{supp-mu-m}  | \Supp(\mu_p^{(2^t)})| > (1 - O(p^{-c})) |H|,\end{equation} where $\mu_H$ is the uniform measure on $H$.
It follows from \eqref{mu-power-group} that
\begin{align*}
1 - O(p^{-c}) \leq \mu_p^{(2^t)}(H) & = (\mu_p^{(2^t - 1)} \ast \mu_p)(H)  =  \\ & = \sum_x \mu_p^{(2^t - 1)}(x) \mu_p (x^{-1} H) \leq \sup_x \mu_p(x H).
\end{align*} 
Thus there is some coset $xH$ such that 
\begin{equation}\label{large-measure-coset} \mu_p(xH) \geq 1 - O(p^{-c}).\end{equation}

\textbf{Step 3.} In this step we use a group-theoretic argument, making use of some slightly specific features of the problem, to upgrade the statement \eqref{large-measure-coset} to $\mu(xH) = 1$, or in other words (recalling Definitions \ref{symplectic-element} and \ref{measure-def}) to show that all $g^{(r)}$, $r \in (\Z/p\Z)^*$, lie in $xH$.

Let $R := \{r \in (\Z/p\Z)^* : g^{(r^{-1})} \in xH\}$. Thus, by \eqref{large-measure-coset},
\begin{equation}\label{R-lower} |R| \geq (1 - O(p^{-c})) (p-1).\end{equation}
Perform the following algorithm to generate distinct elements $r_1, r_2,\dots$ of $R$ as long as possible. Write $S_j := \bigcap_{i \leq j} (xH)^{(r_i)}$. Each $S_j$ is a coset (of some subgroup of $\Sp_8(\Z/p\Z)$) and, no matter how we choose the $r_i$, we have the nesting
\[ S_1 \supseteq S_2 \supseteq \cdots \]
If, at step $j$ of the construction, it is possible to choose $r_{j+1} \in R$ so that $S_{j+1}$ is a \emph{proper} subset of $S_j$ then do so; otherwise, stop.

Note that, as long as the algorithm continues, we have $|S_{j+1}| \leq \frac{1}{2} |S_j|$ (since the $S_j$ are all cosets of subgroups). Therefore, the algorithm stops in at most $O(\log p)$ steps.

When the algorithm finishes, we have $r_1,\dots, r_m \in R$, $m = O(\log p)$, and a coset $S := S_m = \bigcap_{i=1}^m (xH)^{(r_i)}$ (of some subgroup). Note that, since $r_i \in R$, we have $g^{(r_i^{-1})} \in xH$ and so $g \in (xH)^{(r_i)}$, and therefore $g \in S$.

Now set 
\[ R' := r_1^{-1} R \cap \cdots \cap r_m^{-1} R,\]
and suppose that $r \in R'$. Since the algorithm we described stopped at the $m$th stage, we have
\[ S \cap (xH)^{(r_i r)} = S \qquad \mbox{for $i = 1,\dots, m$},\] since otherwise we could take $r_{m+1} := r_i r$ (which would be an element of $R$ by the definition of $R'$). It follows that 
\[ S \cap S^{(r)} = S \cap \bigcap_{i=1}^m (xH)^{(r_i r)} = S.\] That is, if $r \in R'$ then $S = S^{(r)}$. It follows that $S$ is invariant under the entire subgroup of $(\Z/p\Z)^*$ generated by $R'$.  However, 
\[ | R'| \geq 1 - m |(\Z/p\Z)^* \setminus R| > \frac{1}{2}(p-1),\] by \eqref{R-lower} and the fact that $m = O(\log p)$. Therefore the group generated by $R'$ is the whole of $(\Z/p\Z)^*$, and so we have that $S = S^{(r)}$ for all $r \in (\Z/p\Z)^*$. 

We showed earlier that $g \in S$. It now follows that $g^{(rr_1)} \in S$ for all $r \in (\Z/p\Z)^*$. In particular, $g^{(rr_1)} \in (xH)^{(r_1)}$, which implies that $g^{(r)} \in xH$ for all $r \in (\Z/p\Z)^*$, which is what we wanted to show. 

Consequently, all the elements $g^{-(r)} g^{(s)}$ lie in $H$. By definition (Definition \ref{gamma-def}), we have $\Gamma_p \leq H$.

It follows that $\Supp(\mu_p^{(2^t)}) \subseteq H$.  However, we showed in \eqref{supp-mu-m} that $|\Supp(\mu_p^{(2^t)})| > (1 - O(p^{-c})) |H| > \frac{1}{2}|H|$, and therefore the group generated by $\Supp(\mu_p^{(2^t)})$ is all of $H$. However, the group generated by $\Supp(\mu_p^{(2^t)})$ is the group generated by $\Supp(\mu_p^{(2)})$ which, as we remarked earlier, is precisely $\Gamma_p$. Finally, we may conclude that $H = \Gamma_p$. 
Therefore \eqref{ell-2-close} may be rewritten as
\begin{equation}\label{ell-2-close-new} \Vert \mu_p^{(2^t)} - \mu_{\Gamma_p} \Vert \ll p^{-c} \Vert \mu_{\Gamma_p} \Vert.\end{equation}

\textbf{Step 4.} In this step of the argument we upgrade \eqref{ell-2-close-new} to a highly uniform estimate by taking a few extra convolution powers. By Cauchy-Schwarz we have
\begin{align}\nonumber  \Vert \mu_p^{(2^{t+1})} - \mu_{\Gamma_p} \Vert_{\infty} & = \Vert (\mu_p^{(2^t)} - \mu_{\Gamma_p}) \ast  (\mu_p^{(2^t)} - \mu_{\Gamma_p}) \Vert_{\infty} \\ \nonumber & \leq \Vert  \mu_p^{(2^t)} - \mu_{\Gamma_p} \Vert^2 \\ \label{pre-uniform} & \ll p^{-2c} \Vert \mu_{\Gamma_p} \Vert^2 = p^{-2c} |\Gamma_p|^{-1}.\end{align}

However, if $\nu$ is some probability measure on a finite group $\Gamma$ of size $N$ and if \[ \Vert \nu - \mu_{\Gamma} \Vert_{\infty} \leq \frac{\eps}{N}\] then
\begin{align*} |\nu \ast \nu (x) - \frac{1}{N}| &  = |(\nu - \mu_{\Gamma}) \ast (\nu - \mu_{\Gamma})(x) | \\ & \leq \sum_y |\nu(y) - \frac{1}{N}||\nu(y^{-1} x) - \frac{1}{N}|  \leq \frac{\eps^2}{N},\end{align*} that is to say
\[ \Vert \nu^{(2)} - \mu_{\Gamma} \Vert_{\infty} \leq \frac{\eps^2}{N}.\]
Applying this $s$ times to \eqref{pre-uniform} gives
\[ \Vert \mu_p^{(2^{t+1 + s})} - \mu_{\Gamma_p} \Vert_{\infty} \leq (Cp^{-2c})^{2^s}|\Gamma_p|^{-1},\] and so, taking a suitably large $s$, there is some power of two $m= 2^{t + 1 + s} = O(1)$ such that
\begin{equation}\label{p100}\Vert \mu_p^{(m)} - \mu_{\Gamma_p} \Vert_{\infty} < p^{-38},\end{equation} provided (as always) $p$ is sufficiently large. Since we are free to choose any sufficiently large $s = O(1)$, we can make the choice so that $m$ is independent of $p$.

Note that \eqref{p100} is a much stronger version of Proposition \ref{subgroup-spread} in the case $q = p$ a sufficiently large prime.

Since $|\Gamma_p| \leq |\Sp_8(\Z/p\Z)| \ll p^{36}$ it follows from \eqref{p100} that 
\begin{equation}\label{one-p} \mu_p^{(m)} \leq (1 + O(p^{-2})) |\Gamma_{p}|^{-1}\end{equation} pointwise. This one-sided estimate is the only one we will need subsequently.\vspace*{8pt}

\textbf{Step 5.} Finally, we turn to the deduction of Proposition \ref{subgroup-spread} itself. That is, we pass from the case $q = p$ a prime to the general case. Suppose then that $q$ is squarefree, and that all its prime factors are sufficiently large (larger than $p_0(T_{37})$, the quantity appearing in Step 1, is enough). We have a natural homomorphism 
\[ \pi : \Sp_8(\Z/q\Z) \rightarrow \prod_{p | q} \Sp_8(\Z/p\Z).\] By the Chinese remainder theorem, the measure $\mu_q$ pushes forward under $\pi$ to the product $\times_{p | q} \mu_p$: we have $\pi(g^{(r)}) = (g^{(r \mdsub p)})_{p | q}$, and the tuple $(r \md{p})_{p | q}$ takes all values in $\prod_{p | q} (\Z/p\Z)^*$ as $r$ ranges over $(\Z/q\Z)^*$. 

Recall that $\Gamma_q$ is the group generated by the $g^{-(r)} g^{(s)}$, $r, s \in (\Z/q\Z)^*$. These groups also behave nicely under projection, as the following lemma shows. 

\begin{lemma}\label{gam-prod}
Suppose that $q$ is squarefree, and let $\pi : \Sp_8(\Z/q\Z) \rightarrow \prod_{p | q}\Sp_8(\Z/p\Z)$ be the natural isomorphism. Then $\pi(\Gamma_q) = \prod_{p | q} \Gamma_p$.
\end{lemma} 
\begin{proof}
It is easy to see that $\pi(\Gamma_q) \subseteq \prod_{p | q} \Gamma_p$. To see that the two are in fact equal, suppose we have elements $\gamma_p \in \Gamma_p$. For some $N$, we may write $\gamma_p = g^{-(r_{1,p})} g^{(s_{1,p})} \cdots g^{-(r_{N,p})} g^{(s_{N,p})}$ for elements $r_{i,p}, s_{i,p} \in (\Z/p\Z)^*$, that is to say as a word consisting of a product of $N$ of the generators. Note that we can use the same $N$ for each $p$ by padding with exponents $r_{i,p} = s_{i,p}$, if necessary, each of which contributes the identity, and we have also taken advantage of the fact that $(g^{-(r)} g^{(s)})^{-1} = g^{-(s)} g^{(r)}$, which means we do not need to worry about including inverses separately. By the Chinese remainder theorem there are $r_i, s_i \in (\Z/q\Z)^*$ such that $r_i \md{p} = r_{i,p}$, and similarly for the $s_i$. Setting $\gamma := g^{-(r_1)} g^{(s_1)} \cdots g^{(-r_N)} g^{(s_N)} \in \Sp_8(\Z/q\Z)$, we see that $\pi(\gamma) = (\gamma_p)_{p | q}$, as desired.
\end{proof}

\emph{Remark.} We caution that this lemma is a rather specific result. If, for example, we defined $\tilde\Gamma_q$ to be the group generated by the $g^{(r)}$, $r \in (\Z/q\Z)^*$, the same argument would not work (consider, for example, the question of how to find an element whose projection to $\Sp_8(\Z/p\Z)$ is $g$, and whose projection to $\Sp_8(\Z/p'\Z)$ is $g^{-1}$).\vspace*{8pt}

With these facts in hand, we may now complete the proof. Let $m$, a power of two, be as in \eqref{one-p}. Then 

\[ \Vert \mu_q^{(m)}\Vert_{\infty} = \prod_{p | q} \Vert \mu_{p}^{(m)}\Vert_{\infty} \leq \prod_{p | q}(1 + O(p^{-2}))|\Gamma_{p}|^{-1} \ll \prod_{p| q} |\Gamma_{p}|^{-1} =|\Gamma_q|^{-1}.\]
This bound, coupled with the fact that $\Supp(\mu^{(m)}_q) \subset \Gamma_q$, implies Proposition \ref{subgroup-spread}.\vspace*{10pt}

\section{Identifying \texorpdfstring{$\Gamma_q$}{Gammaq}}\label{sec10}

We turn now to the task of proving Proposition \ref{quasirandom-gamma}. The first stage is to actually identify the group $\Gamma_q$ that is to say the subgroup of $\Sp_8(\Z/q\Z)$ generated by the $g^{-(r)} g^{(s)}$, $r,s \in (\Z/q\Z)^*$, in an explicit algebraic form. Recall that the symplectic element $g$ is a particular element of $\Sp_8(\Q)$ associated to the quadratic form $Q$ (see Definition \ref{symplectic-element}). In view of Lemma \ref{gam-prod}, it is enough to consider the prime case $\Gamma_p$. 

To understand this group, it seems best to think of $\Sp_8(\Z/p\Z)$ as embedded in the full group $\GL_8(\Z/p\Z)$ (in the obvious way). It then turns out that $\Gamma_p$ is a conjugate (in $\GL_8(\Z/p\Z)$) of the group $\SL_2(\F_p[\Delta]) \leq \GL_8(\Z/p\Z)$, where (recall) $\Delta = 4b^{-1} a b^{-T} c - I$ is the matrix discriminant of our form $Q$. We will discuss such groups at much greater length in due course. For now, note that $\F_p[\Delta]$ is an algebra over $\F_p$ of dimension at most $4$ (by the Cayley-Hamilton theorem). Generically, this algebra will have dimension exactly $4$ and in this case $\SL_2(\F_p[\Delta])$ (and hence $\Gamma_p$) will be a group of of size $\sim p^{12}$. 

The precise structure of $\SL_2(\F_p[\Delta])$ will depend on how the characteristic polynomial $\rho_{\Delta}$ for $\Delta$ splits over $\F_p$, but it will be a direct product of some groups $\SL_2(\F_{p^n})$ for various $n$.

Suppose henceforth that $Q$ is generic, and that $p$ is large enough that $a, b, \Delta$ are invertible over $\F_p$. In order to examine the elements $g^{-(r)} g^{(s)}$ (which, by definition, generate $\Gamma_p$), one eventually hits upon the idea of looking at what might be called a ``block DUL factorisation'' of $g$. Namely, there is an invertible $A$ and symmetric $B,C$ (all 4-by-4 matrices) such that 
\begin{equation}\label{dul-fact} g = \pmat{A}{0}{0}{A^{-T}} \pmat{I}{B}{0}{I} \pmat{I}{0}{C}{I} .\end{equation} To see that such a factorisation exists is simple: writing $g = \smallpmat{P}{Q}{R}{S}$ where $P = -2b^{-T} c$, $Q = b^{-T}$, $R = 4ab^{-T} c - b$, $S = -2a b^{-T}$, we can take \begin{equation}\label{abc-def} A = S^{-T}, \; B = S^T Q, \; C =  S^{-1} R.\end{equation} 
Note that $S$ is well-defined and invertible (over $\Q$ and over $\F_p$) since both $a$ and $b$ are invertible.

The matrices $B$ and $C$ here will be somewhat important in their own right. We calculate
\begin{equation} \label{BC} B = -2b^{-1} a b^{-T}, \quad C = -2c + \frac{1}{2} b^T a^{-1} b, \quad BC = \Delta.\end{equation} We note that the appearance of $\Delta$ here is the reason for its definition. Since $\Delta$ is assumed invertible over $\F_p$, both $B$ and $C$ are invertible over $\F_p$.

The purpose of looking at the DUL factorisation \eqref{dul-fact} is that it renders the action of dilation easy to understand. Indeed, 
\begin{equation}\label{grs} g^{-(r)}g^{(s)} = \pmat{I}{0}{-rC}{I} \pmat{I}{-r^{-1} B}{0}{I} \pmat{I}{s^{-1} B}{0}{I} \pmat{I}{0}{sC}{I}.\end{equation}
Now set \begin{equation}\label{tau-def}\tau := \pmat{I}{0}{0}{B} \in \GL_8(\Q).\end{equation} This is invertible over $\F_p$ and so may considered as an element of $\GL_8(\Z/p\Z)$. Now observe that for $\lambda \in \F_p^*$ we have
\begin{equation}\label{ul-conjs} \pmat{I}{\lambda B}{0}{I} = \tau^{-1} \pmat{I}{\lambda I}{0}{I} \tau, \quad \pmat{I}{0}{\lambda C}{I} = \tau^{-1} \pmat{I}{0}{\lambda \Delta}{I} \tau.\end{equation}
It follows from this and \eqref{grs} that $g^{-(r)} g^{(s)}$ takes values in the subgroup $\tau^{-1} \SL_2(\F_p[\Delta]) \tau \leq \Sp_8(\Z/p\Z)$.

\emph{Remark.} It is important to note that neither $\tau$ nor $\SL_2(\F_p[\Delta])$ are contained in $\Sp_8(\Z/p\Z)$ in general, although of course the conjugate $\tau^{-1} \SL_2(\F_p[\Delta]) \tau$ is. It turns out that (assuming $Q$ is generic, and for sufficiently large $p$) this is the group $\Gamma_p$. This is the first key result of the section.

\begin{proposition}\label{gen-prop}
Suppose that $Q$ is generic and that $p$ is sufficiently large in terms of $Q$. Then $\Gamma_p = \tau^{-1} \SL_2(\F_p[\Delta]) \tau$.
\end{proposition}

Before turning to the proof, we assemble some lemmas. We will need the following two facts about polynomials. 

\begin{lemma}\label{nullstellensatz}
Let $\eta > 0$. Suppose that $F \in \overline{\F}_p[X,Y]$  has total degree $D$ and that $p > 2D/\eta$. Then
\begin{enumerate}
\item If $F(x,y) = 0$ for at least a proportion $\eta$ of all pairs $(x, y) \in \F_p \times \F_p$, then $F$ is identically zero;
\item If $F(x,y)$ takes values in some subfield $k_0 < \overline{\F}_p$, for at least a proportion $\eta$ of all pairs $(x, y) \in \F_p \times \F_p$, then all the coefficients of $F$ lie in $k_0$. 
\end{enumerate}
\end{lemma}
\begin{proof} (i) This is an immediate consequence of the Schwartz-Zippel lemma, which states that if $F$ is not identically zero then the number of solutions to $F(x,y) = 0$ with $x, y \in \F_p$ is at most $Dp$, which is less than $\eta p^2$ under the assumptions of the lemma.

(ii) We begin with a 1-variable version. If $f(x)$ is a polynomial of degree $D$ which takes values in $k_0$ for $D + 1$ different values of $x \in \F_p$ then it follows from Lagrange interpolation that all the coefficients of $f$ lie in $k_0$. 

Turning to the 2-variable statement we actually want, write $F(X) = \sum_{i=0}^D f_i(X) Y^i$, where $\deg f_i \leq D$. Let $S \subset \F_p \times \F_p$ be the set of pairs $(x,y)$ for which $F(x, y) \in k_0$. For each $x$, let $S_x := \{y : (x,y) \in S\}$. Then there is a set $A \subset \F_p$, $|A| \geq \eta p/2$, such that $|S_x| \geq \eta p/2$ for all $x \in A$. If $x \in A$, the 1-variable result implies (since $p > 2D/\eta$) that all the $f_i(x)$ lie in $k_0$. A second application of the 1-variable result then implies that all the coefficients of each $f_i$ lie in $k_0$. 

We remark that the proof technique for (ii) can also be used for (i) (in fact this is essentially the usual proof of Schwartz-Zippel by induction).
\end{proof}

We will also need a couple of lemmas about subgroups of direct products. Both may be found in \cite[Chapter 1]{serre}. The first result is well-known.

\begin{lemma}[Goursat's lemma]\label{goursat}
Let $G_1, G_2$ be groups. Consider the direct product $G_1 \times G_2$ and let $\pi_i : G_1 \times G_2 \rightarrow G_i$ be the two projection maps. Let $H \leq G_1 \times G_2$ be a subgroup, and suppose that $\pi_i(H) = G_i$ for $i = 1,2$. Then there are normal subgroups $N_i \lhd G_i$ and an isomorphism $\phi : G_1/N_1 \rightarrow G_2/N_2$ such that $H$ has the form $\{(g_1, g_2) \in G_1 \times G_2 : \phi(\overline{g}_1) = \overline{g}_2\}$, where $\overline{g}_i$ is the image of $g_i$ in $G_i/N_i$.
\end{lemma}
\begin{proof}
See, for example, \cite[Proposition 1.6]{serre}.
\end{proof}

The second result is somewhat less well-known and is called Ribet's lemma by Serre \cite{serre}. 

\begin{lemma}[Ribet's Lemma]\label{product-perfect}
Let $G_1,\dots, G_n$ be perfect groups, that is to say equal to their own commutator subgroups. Let $H \leq G_1 \times \cdots \times G_n$ be such that the projection $\pi_{ij}$ of $H$ to $G_i \times G_j$ is surjective for every pair $(i,j)$. Then $H$ is the whole of $G_1 \times \cdots \times G_n$.
\end{lemma}
\begin{proof}
See \cite[Proposition 1.8]{serre}.
\end{proof}

The next lemma, which looks a little \emph{ad hoc}, is in some sense the scalar version of Proposition \ref{gen-prop}, and is the heart of the proof of it.

\begin{lemma}\label{generate-scalar}
Suppose that $p \geq 5$ and that $\theta \in \overline{\F}_p \setminus \{0, -1\}$. Then the matrices 
\[ M_{\theta}(r,s) =  \pmat{1}{0}{-r\theta}{1} \pmat{1}{-r^{-1}}{0}{1}\pmat{1}{s^{-1}}{0}{1}\pmat{1}{0}{s\theta}{1} ,\] as $r,s$ range over $\F_p^*$, generate $\SL_2(\F_p(\theta))$.
\end{lemma}
\begin{proof}
Set $k := \F_p(\theta)$. We use the fact that any proper subgroup of $\SL_2(k)$ has a subgroup of index $C = O(1)$ which is conjugate to a subgroup of one of 
\begin{enumerate}
\item the group of upper triangular matrices or 
\item $\SL_2(k_0)$ for some proper subfield $k_0 < k$. 
\end{enumerate}
See Appendix \ref{sl2-app} for further comments on this fact.

Suppose, then, that the $M_{\theta}(r,s)$ generate a proper subgroup $\Gamma < \SL_2(\F_p(\theta))$. Let $\Gamma' \leq \Gamma$, $[\Gamma : \Gamma'] \leq C$, be a subgroup conjugate to a group of type (1) or (2) above. By the pigeonhole principle there is some coset $\Gamma' x$ containing the elements $g^{-(r)} g$, $r \in R$, for some set $R \subset \F_p^*$ of size at least $\frac{1}{C}(p-1)$. Then $\Gamma'$ contains the elements $g^{-(r)} g^{(s)} = g^{-(r)} g^{-1} (g^{-(s)} g^{-1})^{-1}$ for all $r, s \in R$, that is to say for more than $\eta p^2$ pairs $(r,s)$, where $\eta = (2C^2)^{-1}$.

Now we may explicitly compute that if $P = \smallpmat{a}{b}{c}{d} \in \SL_2(k)$ (here $a,b,c$ are local to the proof of Lemma \ref{generate-scalar}, not the coefficients of $Q$) then the bottom-left entry of $P M_{\theta}(r,s) P^{-1}$ is
\[ \frac{1}{rs} \big( d^2 \theta (1 + \theta) (s^2 r - r^2 s) + c^2 (s-r) + cd\theta (r^2 - s^2)\big).\]
Consider the bracketed expression as a polynomial in $r$ and $s$. If this is to be zero for at least $\eta p^2$ pairs $(r,s) \in \F_p^* \times \F_p^*$ then the first part of Lemma \ref{nullstellensatz} implies that (if $p$ is sufficiently large) all the coefficients of this polynomial vanish, and so $c^2 = cd \theta = d^2 \theta (1 + \theta) = 0$. We cannot have $c = d = 0$ (since $P$ is invertible) and so, since $\theta \neq 0, -1$, possibility (1) is excluded.

Turning to the subfield case (2), we proceed similarly, but now using the second part of Lemma \ref{nullstellensatz}. This implies that all of $c^2 = cd \theta = d^2 \theta (1 + \theta) $ lie in $k_0$. If $cd \neq 0$, it follows that 
\[ 1 + \frac{1}{\theta} = \frac{ c^2 \cdot d^2 \theta (1 + \theta)}{(cd\theta)^2} \in k_0,\]
and so $\theta \in k_0$. This, however, is impossible since $k = \F_p(\theta)$ and $k_0$ is assumed to be a \emph{proper} subfield of $k$.  

If $c = 0$, we consider additionally the top-left entry of $PM(r,s) P^{-1}$, which (when $c = 0$) is
\[ \frac{1}{r} \big( ad (1 + \theta) r - ad \theta s + bd (rs - r^2) \theta (1 + \theta) \big).\]
Note that $ad = \det P = 1$.  Therefore, by the second part of Lemma \ref{nullstellensatz}, $\theta \in k_0$. This is again a contradiction.

Finally if $d = 0$, the top-left entry of $P M(r,s) P^{-1}$ is
\[ \frac{1}{rs} \big( ac (s - r) - bc (1 + \theta ) rs + bc \theta r^2\big).\]
Note that $bc = \det P = -1$, and so again we get $\theta \in k_0$.
\end{proof}

\begin{proof}[Proof of Proposition \ref{gen-prop}]
We must show that the elements $\gamma(r,s) := \tau g^{-(r)} g^{(s)} \tau^{-1}$ generate $\SL_2(\F_p[\Delta])$. From \eqref{grs}, \eqref{ul-conjs} we have
\begin{equation}\label{mrs-mat} \gamma(r,s) = \pmat{I}{0}{-r\Delta}{I} \pmat{I}{-r^{-1}I}{0}{I}\pmat{I}{s^{-1}I}{0}{I}\pmat{I}{0}{s\Delta}{I}.\end{equation}
This should be compared with the definition of $M_{\theta}(r,s)$ in Lemma \ref{generate-scalar}. 

Now we are assuming that $Q$ is generic which, by definition, means that the characteristic polynomial $\rho_{\Delta}(\lambda) = \det (\Delta - \lambda I)$ has four distinct roots in $\overline{\Q} \setminus \{0, -1\}$. Consequently, it will also be the minimum polynomial of $\Delta$ over $\Q$.

Let $\overline{\rho}_{\Delta} \in \F_p[X]$ be the reduction of $\rho_{\Delta}$ modulo $p$. If $p$ is sufficiently large, $\overline{\rho}_{\Delta}$ will have four distinct roots in $\overline{\F}_p \setminus \{-1,0\}$. (The resultant $\Res(\rho_{\Delta}, \rho'_{\Delta}) \in \Q[X]$ is not the zero polynomial, by assumption, and so it is also not the zero polynomial when reduced mod $p$, for $p$ sufficiently large; also $\rho_{\Delta}(0), \rho_{\Delta}(-1) \neq 0$ in $\Q$, and hence in $\F_p$ for $p$ sufficiently large.) Since it has distinct roots, $\overline{\rho}_{\Delta}$ will be the minimal polynomial of $\Delta$ over $\F_p$.

Henceforth, suppose that $p \geq p_0(Q)$. Suppose that the factorisation of $\overline{\rho}_{\Delta}$ into irreducibles polynomials over $\F_p$ is $f_1\cdots f_n$. Then, since $\overline{\rho}_{\Delta}$ has distinct roots in $\overline{\F}_p$, the $f_i$ will be coprime.
For each $i$, let $\alpha_i \in \overline{\F}_p$ be a root of $f_i$ and consider the map
\[\Phi :  \F_p[\Delta] \rightarrow  \prod_{i=1}^n \F_p(\alpha_i)\]
given by
\[ \Phi(F(\Delta)) = (F(\alpha_1),\dots, F(\alpha_n))\] for any $F \in \F_p[X]$. This is a well-defined ring homomorphism: if $F_1(\Delta) = F_2(\Delta)$ then $\overline{\rho}_{\Delta} | F_1 - F_2$ (since $\overline{\rho}_{\Delta}$ is the minimal polyomial of $\Delta$) and hence $f_i | F_1 - F_2$ for each $i$, whence $F_1(\alpha_i) = F_2(\alpha_i)$. We claim that $\Phi$ is injective. If $\Phi(F(\Delta)) = 0$ then for $i = 1,\dots, n$ we have $F(\alpha_i) = 0$ which implies $f_i | F$. Since the $f_i$ are coprime, $\overline{\rho}_{\Delta} | F$ and so $F(\Delta) = 0$.  Both the domain and range of $\Phi$ have size $p^4$ and so it is in fact a ring isomorphism. 

Therefore there are five possible isomorphism types for the ring $\F_p[\Delta]$, namely $\F_{p^4}$ (if $\overline{\rho}_{\Delta}$ is irreducible over $\F_p$), $\F_{p^3} \times \F_p$, $\F_{p^2} \times \F_{p^2}$, $\F_{p^2} \times \F_p \times \F_p$, or $\F_p \times \F_p \times \F_p \times \F_p$ (if $\overline{\rho}_{\Delta}$ splits completely over $\F_p$.) Note that by standard algebraic number theory we can expect all of these possibilities to occur as $p$ varies over primes. 

The map $\Phi$ induces a group isomorphism
\[ \Phi : \SL_2(\F_p[\Delta]) \rightarrow \prod_{i=1}^n \SL_2(\F_p(\alpha_i)).\]
In view of \eqref{mrs-mat}, we have
\[ \Phi(\gamma(r,s)) = (M_{\alpha_i}(r,s))_{i = 1}^n,\] where the $M_{\alpha_i}(r,s)$ are as defined in Lemma\ref{generate-scalar}.
It follows from this and Lemma \ref{generate-scalar} that, if $\Gamma = \langle \gamma(r,s) : r,s \in \F_p^*\rangle = \tau \Gamma_p \tau^{-1}$ is the group generated by the $\gamma(r,s)$, then the projection of $\Phi(\Gamma)$ on to each factor $\SL_2(\F_p(\alpha_i))$ is surjective. When $n = 1$, this is the end of the proof, but we must work a little harder in the other cases.

Let us begin by looking at $H$, the projection of $\Phi(\Gamma)$ to the product $\SL_2(\F_p(\alpha_1)) \times \SL_2(\F_p(\alpha_2))$ of two of the factors (without loss of generality, the first two). Write $\pi_i$, $i = 1,2$ for projection onto each factor. As we have remarked, $\pi_i(H) = \SL_2(\F_p(\alpha_i))$. This allows us to apply Goursat's lemma (Lemma \ref{goursat}). We conclude that there are $N_i \lhd \SL_2(\F_p(\alpha_i))$, and an isomorphism $\phi : \SL_2(\F_p(\alpha_1)) /N_1\rightarrow \SL_2(\F_p(\alpha_2))/N_2$ such that $H = \{(g_1, g_2) \in \SL_2(\F_p(\alpha_1)) \times \SL_2(\F_p(\alpha_2)) : \phi(\overline{g}_1) = \overline{g}_2\}$, where $\overline{g}_i$ denotes reduction mod $N_i$. 

Now the $\SL_2(\F_p(\alpha_i))$ are almost simple: each $N_i$ must be either trivial, $\{ \pm I\}$ or $\SL_2(\F_p(\alpha_i))$. See Appendix \ref{sl2-app}. Moreover, the fact that $\phi$ is an isomorphism, and the fact that $\SL_2(\F_p(\alpha_i))$ is not isomorphic to $\PSL_2(\F_p(\alpha_j))$ (consider cardinalities), means that up to relabelling there are only three essentially different cases, which we consider separately below.

\emph{Case 1.} $N_1 = \SL_2(\F_p(\alpha_1))$. Then $N_2 = \SL_2(\F_p(\alpha_2))$, and $H$ is the whole of the product $\SL_2(\F_p(\alpha_1)) \times \SL_2(\F_p(\alpha_2))$.

\emph{Case 2.} $N_1 = N_2 = \{I\}$. Then $H = \{(x, \phi(x)) : x \in \SL_2(\F_p(\alpha_1))\}$, where $\phi : \SL_2(\F_p(\alpha_1)) \rightarrow \SL_2(\F_p(\alpha_2))$ is some isomorphism. Note that $H$ contains the elements $(M_{\alpha_1}(r,s), M_{\alpha_2}(r,s))$, $r,s \in \F_p^*$, so in this scenario we must have 
\begin{equation}\label{phi-hom} \phi(M_{\alpha_1}(r,s)) = M_{\alpha_2}(r,s)\end{equation} for all $r,s$. By looking at cardinalities, the fields $\F_p(\alpha_1)$ and $\F_p(\alpha_2)$ must be isomorphic, so to ease notation we may suppose that $\alpha_2 \in \F_p(\alpha_1)$. 

Now it is known (see Appendix \ref{sl2-app}) that the automorphism group of $\SL_2(k)$ is generated by conjugation by elements of $\GL_2(k)$ and field automorphisms. Therefore for some $P  \in \GL_2(\F_p(\alpha_1))$ and for some field automorphism $\sigma$ of $\F_p(\alpha_1)$ we have
\[ \phi(M_{\alpha_1}(r,s)) = P M_{\sigma(\alpha_1)}(r,s) P^{-1}.\] Comparing with \eqref{phi-hom} gives
\begin{equation}\label{prs} P M_{\sigma(\alpha_1)}(r,s) = M_{\alpha_2}(r,s)  P\end{equation} for all $r,s$. Writing $P = \smallpmat{a}{b}{c}{d}$ and comparing top left entries gives, writing $\theta = \sigma(\alpha_1)$ and $\theta' = \alpha_2$,
\begin{equation}\label{abc} a (1 + \theta - \frac{s}{r} \theta) + b(s - r) \theta (1 + \theta) = a( 1 + \theta' - \frac{s}{r} \theta') + (\frac{1}{s} - \frac{1}{r}) c\end{equation} for all $r,s \in \F_p^*$. By Lemma \ref{nullstellensatz} (and since $\theta \neq 0, -1$) we have $b = c = 0$, thus $a \neq 0$ and $\theta = \theta'$. That is, $\sigma(\alpha_1) = \alpha_2$, and so $\alpha_1, \alpha_2$ have the same minimal polynomial over $\F_p$. This is a contradiction, since we assumed that the minimal polynomials of $\alpha_1, \alpha_2$ over $\F_p$ (the $f_i$, that is to say the factors of the minimal polynomial of $\Delta$ over $\F_p$) are coprime.

\emph{Case 3.} $N_i = N_j = \{ \pm I\}$. We reprise the argument from Case 2, only now we must allow a sign error. Included in the classification of automorphisms of $\PSL_2(\F_p(\alpha_1))$ (see Appendix \ref{sl2-app}) is the fact that such automorphisms lift to automorphisms of $\SL_2(\F_p(\alpha_1))$. Therefore $H = \{(x, \eps(x)\phi(x)) : x \in \SL_2(\F_p(\alpha_1)), \eps(x) \in \pm I\}$ for some isomorphism $\phi : \SL_2(\F_p(\alpha_1)) \rightarrow \SL_2(\F_p(\alpha_2))$. We may now proceed as before but with an additional sign error, thus \eqref{prs} becomes
\begin{equation}\label{prs-signed} P M_{\sigma(\alpha_1)}(r,s) = \eps_{r,s} M_{\alpha_2}(r,s) P\end{equation} for all $r,s \in \F_p^*$ and for some choice of signs $\eps_{r,s} \in \{\pm 1\}$. If $\eps_{r,s} = 1$ for at least half of all pairs $(r,s) \in \F_p^* \times \F_p^*$ then we are done, exactly as before (taking $\eta = \frac{1}{4}$ in Lemma \ref{nullstellensatz}). If $\eps_{r,s} = -1$ for at least half of all pairs $(r,s) \in \F_p^* \times \F_p^*$ then \eqref{abc} is modified to 
\begin{equation}\label{abc-neg} a (1 + \theta - \frac{s}{r} \theta) + b(s - r) \theta (1 + \theta) = - a( 1 + \theta' - \frac{s}{r} \theta') - (\frac{1}{s} - \frac{1}{r}) c,\end{equation} for half of all pairs $(r,s) \in \F_p^* \times \F_p^*$. From this we conclude that $b = c = 0$, hence $a \neq 0$ and so both $\theta = - \theta'$ and $1 + \theta = -(1 + \theta')$. This is impossible.\vspace*{8pt}

Since only Case 1 in the above analysis did not lead to a contradiction (and since we can replace $\{1,2\}$ by any pair $\{i,j\}$), we have now shown that the projection of $\Phi(\Gamma)$ to the product $\SL_2(\F_p(\alpha_i)) \times \SL_2(\F_p(\alpha_j))$ of any pair of factors is surjective. Proposition \ref{gen-prop} now follows from Lemma \ref{product-perfect}, together with the fact (see Appendix \ref{sl2-app}) that all the factors $\SL_2(\F_p(\alpha_i))$ are perfect.
\end{proof}

\section{Quasirandomness of \texorpdfstring{$\rho|_{\Gamma_q}$}{rhoGammaq}.}\label{sec11}

We turn now to the proof of Proposition \ref{quasirandom-gamma} itself. Let us begin by recalling the statement.

\begin{proposition}[Proposition \ref{quasirandom-gamma}]
Let $Q$ be a generic quadratic form. Then there is some $p_0(Q)$ such that the following is true. Let $\rho$ be the Weil representation on $\Sp_8(\Z/q\Z)$ \textup{(}as given in Proposition \ref{weil-prop}\textup{)}. Suppose that $q$ is squarefree and has all prime factors greater than $p_0(Q)$. Then $\rho|_{\Gamma_q}$ splits into irreducible subrepresentations of dimensions $\geq q^{1 - o(1)}$. 
\end{proposition}

We begin by reducing to the prime case. Write $q = p_1 \cdots p_n$. The representation $\rho$ is constructed in Appendix \ref{weil-app} as a tensor product $\otimes_{i=1}^n \tilde\rho_i$, where $\tilde\rho_i : \Sp_8(\Z/p_i\Z) \rightarrow \U(\ell^2((\Z/p_i\Z)^4)$ is a twisted version of the mod $p_i$ Weil representation $\rho_i$, given by $\tilde \rho_i (g) = \rho_i(g^{\sigma_i})$ where $\smallpmat{M_1}{M_2}{M_3}{M_4}^{\sigma_i} = \smallpmat{M_1}{\lambda_i M_2}{\lambda_i^{-1} M_3}{M_4}$ and $\lambda_i = \prod_{j \neq i} p_i$.  This tensor product may be realised on $\ell^2((\Z/q\Z)^4)$ by $\rho(g) f = \prod_{i=1}^n \tilde\rho_i(g) f_i(x)$ for ``pure tensors'' $f(x) = \prod_{i=1}^n f_i(x)$, where $f_i$ factors through the projection $\pi_i : (\Z/q\Z)^4 \rightarrow (\Z/p_i\Z)^4$ (see Appendix \ref{weil-app} for a discussion of the notation here). Consider the restriction to $\Gamma_q \cong \prod_{i=1}^n \Gamma_{p_i}$. The decomposition of $\rho|_{\Gamma_q}$ into irreducibles is then given by decomposing each $\ell^2((\Z/p_i\Z)^4)$ into irreducible $\tilde\rho_i|_{\Gamma_{p_i}}$-invariant subspaces $V_i$ and taking tensor products. (Here we use the fact that if $V_i$ is an irreducible $G_i$ representation then $\bigotimes_{i=1}^n V_i$ is an irreducible $\times_{i=1}^n G_i$-representation, which is a standard fact of representation theory. See for instance \cite[Theorem 19.18]{james-liebeck}.)

Now each $\Gamma_{p_i}$ is a direct product of groups $\SL_2(\F_{p_i^j})$, and therefore by Appendix \ref{sl2-app} any irreducible representation of $\Gamma_{p_i}$ is either trivial, or has dimension at least $\frac{1}{2}(p_i-1)$. That is, if $\dim V_i \neq 1$ then $\dim V_i \geq \frac{1}{2}(p_i-1)$. Consequently, if for all $i$ the representation $\tilde\rho_i|_{\Gamma_{p_i}}$ has no invariant vector (that is, 1-dimensional invariant subspace) then $\dim \rho  \geq \prod_{p | q} \frac{1}{2}(p-1) = q^{1 - o(1)}$, as desired. 

This reduces the task of proving Proposition \ref{quasirandom-gamma} to the following, which is the final task for the main part of the paper.

\begin{proposition}\label{key-nontrivial}
Suppose that $Q$ is generic and that $p$ is sufficiently large in terms of $Q$. Let $\rho : \Sp_8(\Z/p\Z)  \rightarrow \U(\ell^2((\Z/p\Z)^4))$ be the mod $p$ Weil representation. Let $r \in (\Z/p\Z)^*$, and let $\tilde \rho$ be the twist of $\rho$ by dilation by $r$, that is to say $\tilde\rho(x) = \rho(x^{(r)})$. Then $\tilde \rho|_{\Gamma_p}$ has no nontrivial invariant vector.\end{proposition}
\begin{proof}
We will show that the conclusion holds under the assumption that $\Delta$ has distinct eigenvalues and is invertible over $\F_p$. This includes all sufficiently large primes $p$. Indeed $Q$ is generic, so by definition $\Delta$ has distinct eigenvalues and is invertible over $\Q$. Therefore the same is true over $\F_p$, $p$ sufficiently large, for the reasons detailed at the start of the proof of Proposition \ref{gen-prop}.

Suppose from now on that $\Delta$ has distinct eigenvalues and is invertible over $\F_p$. Since $\Gamma_p$ is invariant under the dilation $\gamma \mapsto \gamma^{(r)}$, it suffices to consider the case $\tilde\rho = \rho$. Suppose, then, that
\begin{equation}\label{invariance-cond}  \rho(\gamma) f = f\end{equation} for all $\gamma \in \Gamma_p$. Our aim is to show that $f$ is identically zero. To examine the condition \eqref{invariance-cond} we will look at the following particular elements $\gamma$, where $B, C \in \Mat_4(\F_p)$ are the specific symmetric matrices described in \eqref{BC}:
\begin{enumerate}
\item the upper triangular elements $u(MB) = \smallpmat{I}{MB}{0}{I}$, where $M \in \F_p[\Delta]$;
\item the lower triangular elements $l(CM) = \smallpmat{I}{0}{CM}{I}$, where $M \in \F_p[\Delta]$;
\item the diagonal elements $s(\lambda I) = \smallpmat{\lambda I}{0}{0}{\lambda^{-1} I}$,  $\lambda \in \F_p^*$. 
\end{enumerate}
Now that we know from Proposition \ref{gen-prop} that $\Gamma_p = \tau^{-1} \SL_2(\F_p[\Delta]) \tau$ (where $\tau$ is defined in \eqref{tau-def}), so one may easily check using \eqref{ul-conjs} that all of these elements do lie in $\Gamma_p$. 

Now we already have formulae for the actions of the elements in (2) and (3), directly from Proposition \ref{weil-prop}. 
Namely, 
\begin{equation}\label{first-inv} \rho(l(CM)) f(x) = \xi_M e_p(-\frac{1}{2} x^T CM x) f(x)\end{equation}for some unit complex number $\xi_M = \xi(l(CM))$ (we do not care exactly what this is)
and
\begin{equation}\label{second-inv}  \rho(s(\lambda I)) f(x) = \xi'_{\lambda} f(\lambda^{-1} x)\end{equation} for some unit complex number $\xi'_{\lambda}$.
If \eqref{invariance-cond} holds, it follows from \eqref{first-inv} that for each $M$
\[ \Supp(f) \subset \{ x : x^T CM x = t_M \}\] for some parameters $t_M$ (in fact satisfying $e_p(-\frac{1}{2} t_M) = \xi_M$). Now \eqref{second-inv} tells us that $\Supp(f(x)) = \Supp(f (\lambda^{-1} x))$ for all $\lambda \in \F_p^*$, therefore (taking $\lambda \neq \pm 1$) we see that in fact all the $t_M$ must be zero, that is
\begin{equation}\label{support-cond} \Supp(f) \subset \{ x : x^T CM x = 0 \} \quad \mbox{for all $M \in \F_p[\Delta]$}. \end{equation}

To get a formula for the upper triangular action, we note the identity
\[ u(W) = J^{-1} l(-W) J\] for any matrix $W$, where as usual $J = \smallpmat{0}{I}{-I}{0}$.  Thus if $f$ is invariant under $\rho(u(MB))$ then $\rho(J) f$ is invariant under $\rho(l(-MB))$. As above, this implies that $\Supp (\rho(J) f) \subset \{ x : x^T MB x = t'_M\}$. Since $J^{-1} s(\lambda I) J = s(\lambda^{-1} I)$, we see that $\rho(J) f$ is invariant (up to multiplication by a scalar) by dilation $x \mapsto \lambda^{-1} x$. Therefore all of the $t'_M$ are in fact zero.  Moreover, from Proposition \ref{weil-prop} we know that $\rho(J) f$ is the (normalised) Fourier transform of $f$ times a scalar, and so we come to the conclusion that 

\begin{equation}\label{support-cond-dual} \Supp(\hat{f}) \subset \{ x : x^T MB x = 0 \} \quad  \mbox{for all $M \in \F_p[\Delta]$} .\end{equation}

Conditions \eqref{support-cond} and \eqref{support-cond-dual} seem highly incompatible and should, for example, violate the uncertainty principle under reasonable assumptions. However, a proof seems not to be completely straightforward (and indeed the statement fails in sufficiently degenerate situations, for example if we were to allow $\Delta = 0$). 

Let us begin the argument. Consider the bilinear form $\phi(x,y) = x^T C y$ on $\overline{\F}_p^4 \times \overline{\F}_p^4$. Since $BC = \Delta$ and $\Delta$ is invertible over $\F_p$, $C$ is invertible over $\F_p$ and so this is a non-degenerate form. We note that $\Delta$ is self-adjoint with respect to $\phi$:
\[ \phi(x, \Delta y) = x^T CBC y,\]
\[ \phi(\Delta x, y) = x^T \Delta^T C y = x^T CBC y\] (note that $B, C$ are both symmetric). 
By the usual argument, this means that eigenvectors of $\Delta$ with distinct eigenvalues are orthogonal with respect to $\phi$. Indeed, if $\Delta v_1 = \lambda_1 v_1$ and $\Delta v_2 = \lambda_2 v_2$ then
\[ \lambda_2 \phi(v_1, v_2) = \phi(v_1, \Delta v_2) = \phi(\Delta v_1, v_2) = \lambda_1 \phi(v_1, v_2).\]
Now since $\Delta$ has distinct eigenvalues, it is diagonalisable over $\overline{\F}_p$. Suppose that $v_i$ are eigenvectors with (distinct) eigenvalues $\lambda_i$, $i = 1,2,3,4$. These are a basis for $\overline{\F}_p^4$. Write $x \in \overline{\F}_p^4$ as $\sum_{i=1}^4 x_i v_i$. Then
\begin{equation}\label{xcx} x^T C \Delta^j x = \phi(x, \Delta^j x) = \sum_{i=1}^4 x_i^2 \lambda_i^j \phi(v_i, v_i).\end{equation}
Note that $\phi(v_i, v_i) \neq 0$ (if it was, $v_i$ would be orthogonal with respect to $\phi$ to all of $v_1, v_2, v_3, v_4$ and hence to all of $\overline{\F}_p^4$, contrary to the fact that $\phi$ is non-degenerate). 

Therefore the matrix with $(i,j)$-entry $\lambda_i^j \phi(v_i, v_i)$ ($1 \leq i \leq 4, 0 \leq j \leq 3$) is non-singular, its determinant being a non-zero multiple $\prod_{i=1}^4 \phi(v_i, v_i)$ of a certain Vandermonde determinant. It follows from \eqref{xcx} that if $x^T C \Delta^j x = 0$ for $j = 0,1,2,3$ then $x = 0$, and so any $f$ satisfying \eqref{support-cond} is supported only at zero.

Noting that $x^T \Delta^{j+1} B x = (Bx)^T C \Delta^{j} (Bx)$, we can also conclude that if $x^T \Delta^{j+1} B x = 0$ for $j = 0,1,2,3$ then $x = 0$, and so any $f$ satisfying \eqref{support-cond-dual} has $\hat{f}$ supported only at zero.

These two facts about $f$ are completely incompatible, unless $f$ is identically zero: if $f$ is supported at zero, $\hat{f}$ is in fact constant on $(\Z/p\Z)^4$.

This completes the proof of Proposition \ref{key-nontrivial}, and hence that of Proposition \ref{quasirandom-gamma}.
\end{proof}

All of the main results in the paper are now established.\vspace*{6pt}

\emph{Remark.} It is in fact possible to show (under the assumptions on $\Delta$ in force throughout this section) that $\rho|_{\Gamma_p}$ is isomorphic to a Weil representation of $\SL_2(\F_p[\Delta])$ on $\ell^2(\F_p[\Delta])$, by giving an explicit intertwining map. This fact can be used to give an alternative proof of Proposition \ref{key-nontrivial} which, while more natural than the \emph{ad hoc} arguments presented here, requires quite a bit more setting up. We intend to give a full account in future work.

\appendix

\section{Facts about \texorpdfstring{$\SL_2(k)$}{SL2(k)}}\label{sl2-app}

We collect various well-known facts about $\SL_2(k)$, $k$ a finite field, which we used in the main text. For our purposes, ``rough'' versions of these facts (passing to subgroups of index $O(1)$, etc) would be quite sufficient but we use the precise versions when sufficiently clean results are relatively easily-available. 

\begin{proposition}\label{sl2-facts}
Let $k$ be a finite field of odd characteristic. Then
\begin{enumerate}
\item The smallest nontrivial complex representation of $\PSL_2(k)$ has dimension at least $\frac{1}{2}(|k| - 1)$. 
\item If $k$ has order at least $5$ then $\SL_2(k)$ is perfect.
\item Any proper subgroup of $\SL_2(k)$ has a subgroup of index at most $C$ which is conjugate to a subgroup of one of  \textup{(i)} the group of upper triangular matrices or \textup{(ii)} $\SL_2(k_0)$ for some proper subfield $k_0 < k$. 
\item Every automorphism of $\PSL_2(k)$ or $\SL_2(k)$ is a composition of a conjugation by elements of $\GL_2(k)$, and a power of the Frobenius automorphism of $k$.
\end{enumerate}

\end{proposition}
\begin{proof}
(1) goes back well over a century, to Jordan and Schur. For a nice and easy-to-access discussion, see Prasad's notes \cite{prasad}.

(2)  See Lang \cite[Chapter XIII, Theorem 8.3]{lang}.

(3) The rough statement given here, which suffices for our purposes is \cite[Theorem 5.2.7]{tao-expansion}. As one would expect, a detailed classification of maximal subgroups of $\SL_2(k)$ has been known for more than a century. It is somewhat complicated; the details, as well as references to the original papers, may be found in \cite{oking}.

(4) This is certainly well-known. A standard reference is \cite[3.2]{steinberg-1}. See also \cite{steinberg-2}. The MathOverflow post \cite{auto-overflow} is helpful in navigating these papers. 

\end{proof}

\section{Weil representation of \texorpdfstring{$\Sp_8(\Z/q\Z)$}{Sp8(Z/qZ)}}\label{weil-app}

In the section we construct the representation of $\rho : \Sp_8(\Z/q\Z) \rightarrow \U(\ell^2((\Z/q\Z)^4)$ which we have been using throughout the paper, and whose properties are detailed in Proposition \ref{weil-prop}.  Suppose that $q = p_1 \cdots p_n$ is squarefree and odd. We assume the existence, for each $i$, of the Weil representations $\rho_i : \Sp_8(\Z/p_i\Z) \rightarrow \U(\ell^2((\Z/p_i\Z)^4))$, the construction of which is given in detail in \cite{neuhauser} and shown to satisfy the properties of Proposition \ref{weil-prop} (and in fact that paper gives details of the multiplier $\xi$, whose precise properties are unimportant in this paper). Since $\Sp_8(\Z/q\Z) \cong \prod_{i=1}^n \Sp_8(\Z/p_i\Z)$, one thinks of looking at the (exterior) tensor product $\rho = \otimes_{i=1}^n \rho_{i}$. However, this turns out to need a small modification.

For each $i$, denote by $\sigma_i : \Sp_8(\Z/p_i\Z) \rightarrow \Sp_8(\Z/p_i\Z)$ the automorphism defined by $\smallpmat{M_1}{M_2}{M_3}{M_4}^{\sigma_i} := \smallpmat{M_1}{\lambda_i M_2}{\lambda_i^{-1} M_3}{M_4}$, where $\lambda_i := \prod_{i \neq j} p_i$. Note that this is, in fact, one of the dilates we considered earlier (see \eqref{gr-def}), but with parameter $r = \lambda_i^{-1}$. It may not be an inner automorphism (this depends on whether or not $\lambda_i$ is a square mod $p_i$).  Each $\rho_{i}$ may be twisted by $\sigma_i$ to give a representation $\tilde\rho_i : \Sp_8(\Z/p_i\Z) \rightarrow \U(\ell^2((\Z/p_i\Z)^4))$, defined by $\tilde\rho_i(g) f := \rho(g^{\sigma_i}) f$. This will be isomorphic to $\rho_i$ if $\lambda_i$ is a square mod $p_i$, but not otherwise. However, we will not need this last fact.

We now construct $\rho$ as the tensor product $\bigotimes_{i=1}^n \tilde\rho_i$, which we shall shortly show how to realise concretely in $\ell^2((\Z/q\Z)^4)$. From here on, we will abuse notation by omitting explicit notation for projection maps from $(\Z/q\Z)^4 \rightarrow (\Z/p_i\Z)^4$, from $\Sp_8(\Z/q\Z)$ to $\Sp_8(\Z/p_i\Z)$, or from $\Mat_4(\Z/q\Z)$ to $\Mat_4(\Z/p_i\Z)$ when the domain is clear from context. Thus, for example, for functions $f_i \in \ell^2((\Z/p_i\Z)^4)$ we define the ``pure tensor'' $f(x) := \prod_{i=1}^n f_i(x)$, but it would be more correct, though cumbersome, to write $\prod_{i=1}^n f_i(\pi_i(x))$ where $\pi_i : (\Z/q\Z)^4 \rightarrow (\Z/p_i\Z)^4$ is the natural projection.

For a pure tensor $f$ as above define
\begin{equation}\label{tens-explicit-rho} \rho(g) f(x) = \prod_{i=1}^n \tilde\rho_{i}(g) f_i(x) = \prod_{i=1}^n \rho_{i}(g^{\sigma_i}) f_i(x).\end{equation} 
This is well-defined by the universal property of tensor products. 

We now turn to the verification of the properties stated in Proposition \ref{weil-prop}. 
Recall that the properties to be established are as follows (for some unit complex numbers $\xi(\cdot )$):

\begin{equation}\label{seq} \rho(s(E)) f(x) = \xi(s(E)) f(E^{-1} x)\end{equation}
for $s(E) := \smallpmat{E}{0}{0}{E^{-T}}$ with $E$ invertible;

\begin{equation}\label{jeq} \rho(J) f(x) =  \xi(J) q^2 \E_{y \in (\Z/q\Z)^4} f(y) e(x^T y);\end{equation}

\begin{equation}\label{ueq} \rho(l(W)) f(x) = \xi(l(W))e_q(-\frac{1}{2} x^T W x)f(x)\end{equation}
for $l(W) := \smallpmat{I}{0}{W}{I}$ with $W$ symmetric.

We will also, of course, be using the corresponding properties for the prime case $\rho_i$. To avoid confusion, we write $\xi_i(\cdot)$ for the corresponding unit complex numbers. It is enough to check \eqref{seq}, \eqref{jeq} and \eqref{ueq} for pure tensors $f$.

\emph{Proof of \eqref{seq}.} We have
\begin{align*} \rho(s(E)) f(x) & = \prod_{i=1}^n \rho_i (s(E)^{\sigma_i}) f_i(x)  = \prod_{i=1}^n \rho_i (s(E)) f_i(x) \\ & = \prod_{i=1}^n \xi_i(s(E)) f_i(E^{-1} x) = \xi(s(E)) f(E^{-1} x),\end{align*} where $\xi(s(E)) := \prod_{i=1}^n \xi_i(s(E))$. This establishes \eqref{seq}.\vspace*{8pt}

For the remaining two parts, we will need the relation
\begin{equation} \label{eqep-rel} e_q(\lambda t) = \prod_{i=1}^n e_p(t),\end{equation} where $\lambda = \sum_{i=1}^n \lambda_i$.\vspace*{8pt}

\emph{Proof of \eqref{jeq}.} We have
\begin{align*} \rho(J) f(x) & = \prod_{i=1}^n \rho_i(J^{\sigma_i}) f_i(x) = \prod_{i=1}^n \rho_i(s(\lambda_i I) J) f_i(x) \\ & =  q^2 \xi(J) \prod_{i=1}^n \E_{y_i \in (\Z/p_i\Z)^4} f_i(y_i) e_{p_i}(\lambda_i^{-1} x^T y_i),\end{align*}
where $\xi(J) := \prod_{i=1}^n \xi_i(s(\lambda_i I)) \xi_i(J)$. Applying \eqref{eqep-rel} gives \[ \prod_{i=1}^n e_{p_i}(\lambda_i^{-1} x^T y_i) = e_q(x^T y),\] where $y \in (\Z/q\Z)^4$ reduces to $y_i$ in $(\Z/p_i\Z)^4$, and the claim follows.\vspace*{8pt}

\emph{Proof of \eqref{ueq}.} We have
\begin{align*}
\rho(l(W)) f(x) & = \prod_{i=1}^n \rho_i(l(W)^{\sigma_i}) f_i(x) = \prod_{i=1}^n \rho_i (l(\lambda_i^{-1} W)) f_i (x) \\ & = \xi(l(W)) \prod_{i=1}^n e_{p_i}(-\frac{1}{2} x^T \lambda_i^{-1} W x) f_i(x),
\end{align*}
where $\xi(l(W)) := \prod_{i=1}^n \xi_i (l(\lambda_i^{-1} W))$. The claim again follows using \eqref{eqep-rel}.

\section{Almost-invariant measures}\label{appc}

We invite the reader to recall the notation concerning probability measures on finite groups $G$ as described at the start of Section \ref{sec8}. In particular, $\Vert \cdot \Vert$ means the $2$-norm with respect to the counting measure on $G$, $\Vert \mu \Vert = \big( \sum_{x \in G} \mu(x)^2 \big)^{1/2}$. We extend the notation $\mu_H$ for the uniform measure on a subgroup $H$ to arbitrary sets: thus if $A \subset G$ is a finite set then we write $\mu_A$ for the uniform measure on $A$, that is to say the measure which puts weight $|A|^{-1}$ on each point of $A$.

We have the following instance of Young's inequality.

\begin{lemma}[Young's inequality]\label{young}
Let $\mu, \nu$ be two probability measures on $G$. Then we have $\Vert \mu \ast \nu \Vert \leq \Vert \mu \Vert$.
\end{lemma}
\begin{proof}
We have
\[ \mu \ast \nu(x) = \sum_{y \in G} \big( \mu(y)^2 \nu(y^{-1}x) \big)^{1/2} \nu(y^{-1}x)^{1/2}.\]
By Cauchy-Schwarz and the fact that $\nu$ is a probability measure,
\[ \mu \ast \nu(x)^2 \leq \sum_{y \in G} \mu(y)^2 \nu(y^{-1}x).\]
Finally, summing over $x \in G$ gives the result.
\end{proof}
In the main text we required a statement about almost equality here in the case that $\mu = \nu$ and both are symmetric. The actual statement we quoted in the main text is Corollary \ref{stability-young-cor} below, but the heart of it is Lemma \ref{stable-young}. This result should be thought of as ``well-known'', but it is hard to give a precise reference. The basic idea of the proof goes back to Fournier \cite{fournier}; see \cite[Proposition 5.4]{eisner-tao} for the abelian case or \cite[Appendix A]{BGGT} for a closely related result. 

\begin{lemma}\label{stable-young}
Let $\mu$ be a symmetric probability measure on a finite group $G$. Let $\eps$ be sufficiently small positive constant. Suppose that $\Vert \mu \ast \mu \Vert \geq (1 - \eps ) \Vert \mu \Vert^2$. Then there is a subgroup $H \leq G$ such that $\Vert \mu - \mu_H \Vert \ll \eps^c \Vert \mu \Vert$. \end{lemma}
\begin{proof}
We assume throughout the proof that $\eps$ is sufficiently small. Set
\begin{equation}\label{a-def} A := \{x : \mu \ast \mu(x) \geq (1 - \eps^{1/2}) \Vert \mu \Vert^2\}.\end{equation} Then $A$ is a symmetric set, and we have (using the pointwise bound $\mu \ast \mu(x) \leq \Vert \mu \Vert^2$)
\begin{align*}
(1 - \eps) \Vert \mu \Vert^2 & \leq \Vert \mu \ast \mu \Vert^2 = \sum_x \mu \ast \mu(x)^2 \\ & \leq \Vert \mu \Vert^2 \mu \ast \mu(A) + (1 - \eps^{1/2}) \Vert \mu \Vert^2 \mu \ast \mu(A^c).
\end{align*}
Writing $\mu \ast \mu(A) = 1 - \delta$, so that $\mu\ast \mu(A^c) = \delta$, this rearranges to give $\delta \leq \eps^{1/2}$, that is to say $\mu \ast \mu(A) \geq 1 - \eps^{1/2}$.
Since 
\[ \mu \ast \mu(A) = \sum_y \mu(y) \mu(A y^{-1}),\] it follows that there is some $B := A y^{-1}$ such that $\mu (B) \geq 1 - \eps^{1/2}$. Since $\mu$ is symmetric, $\mu(B^{-1}) \geq 1 - \eps^{1/2}$. Therefore, setting $S := B \cap B^{-1}$, we see that $S$ is symmetric and $\mu(S) \geq 1 - 2\eps^{1/2}$. From \eqref{a-def} we have
\[ |S| \leq |A| \leq (1 + 2 \eps^{1/2}) \Vert \mu \Vert^{-2}.\] 
Now we have
\[ \Vert \mu - \mu_S \Vert^2 = \Vert \mu \Vert^2 - \frac{1}{|S|}(2 \mu(S) - 1) \leq \Vert \mu \Vert^2 (1 - \frac{1 - 4 \eps^{1/2}}{1 + 2 \eps^{1/2}}) < 8 \eps^{1/2} \Vert \mu \Vert^2\]
and so
\begin{equation}\label{mu-mus} \Vert \mu - \mu_S \Vert = O(\eps^{1/4})\Vert \mu \Vert.\end{equation}
Note that this implies 
\begin{equation}\label{compare} \frac{1}{2}\Vert \mu \Vert \leq \Vert \mu_S \Vert \leq 2 \Vert \mu \Vert.\end{equation}
Writing $\mu_S = \mu + (\mu_S - \mu)$ and expanding and using the triangle inequality, we have
\[ \Vert \mu_S \ast \mu_S\Vert \geq \Vert \mu \ast \mu\Vert - 2 \Vert \mu \ast (\mu_S - \mu)\Vert - \Vert (\mu_S - \mu) \ast (\mu_S - \mu)\Vert.\] By Young's inequality, \eqref{mu-mus}, \eqref{compare} and the assumption of the lemma it follows that 
\begin{equation}\label{s-assump}  \Vert \mu_S \ast \mu_S \Vert \geq (1 - O(\eps^{1/4}) \Vert \mu \Vert \geq (1 - O(\eps^{1/4})) \Vert \mu_S \Vert.\end{equation}
At this point we have essentially reduced the proof of the lemma to the case of a uniform measure on a set. Equation \eqref{s-assump} is equivalent to the statement that the number of multiplicative quadruples $s_1 s_2 = s_3 s_4$ in $S$ is $(1 - O(\eps^{1/4})|S|^3$. This is a well-known situation and (for example) Fournier \cite{fournier} implies that there is a subgroup $H$ such that $|S \triangle H| = O(\eps^c) \min(|H|, |S|)$.
Therefore
\[ \Vert \mu_S - \mu_H \Vert^2 = \frac{|S \triangle H|}{|S||H|} \ll \eps^c \Vert \mu_S \Vert^2 \ll \eps^c \Vert \mu \Vert^2.\]
The result follows from this, \eqref{mu-mus} and the triangle inequality. 
\end{proof}

Finally we give the result actually quoted in the main text.

\begin{corollary}\label{stability-young-cor}
Let $\mu$ be a symmetric probability measure on a finite group $G$. Let $\eps$ be a sufficiently small constant. Suppose that $\Vert \mu \ast \mu \Vert \geq (1 - \eps ) \Vert \mu \Vert$. Then there is a subgroup $H \leq G$ such that $\Vert \mu - \mu_H\Vert \ll \eps^c \Vert \mu_H \Vert$, $\mu(H) \geq 1 - O(\eps^c)$ and $|\Supp(\mu)| \geq (1 - O(\eps^c))|H|$. 
\end{corollary}
\begin{proof}
Let $H$ be as in the conclusion of Lemma \ref{stable-young}. The first statement is just the conclusion of Lemma \ref{stable-young}. From it we deduce
\begin{equation}\label{eq-c-2} \Vert \mu \Vert =(1 + O(\eps^c))\Vert \mu_H \Vert = (1 + O(\eps^c)) |H|^{-1/2}.\end{equation} In particular, 
\begin{equation}\label{eq-c-3} \langle \mu, \mu_H\rangle = \frac{1}{2} \big( \Vert \mu \Vert^2 + \Vert \mu_H \Vert^2 - \Vert \mu - \mu_H\Vert^2) = (1 + O(\eps^c))|H|^{-1} .\end{equation}
Now we have $\langle \mu, \mu_H\rangle = |H|^{-1} \mu(H)$, and so the second statement follows.

For the third statement, write $A := \Supp(\mu) \cap H$. By Cauchy-Schwarz we have
\[ \langle \mu, \mu_H\rangle = \frac{1}{|H|} \sum_x \mu(x) 1_A(x) \leq \frac{1}{|H|} \Vert \mu \Vert |A|^{1/2}.\]
The required bound now follows by comparing this with \eqref{eq-c-2} and \eqref{eq-c-3}.
\end{proof}

\end{document}